\documentclass{article}

\usepackage[
	a4paper,
	margin=25mm
]{geometry}
\usepackage{setspace}
\usepackage{sectsty} 
\usepackage{titlesec}

\usepackage{amsmath, amssymb, amsthm}
\usepackage{mathtools}
\usepackage{bbm}

\usepackage[
	setpagesize=false,
	colorlinks=true,
	linkcolor=blue,
	pdfencoding=auto,
	psdextra,
]{hyperref}
\usepackage{cleveref}

\usepackage{etoolbox} 
\usepackage[shortlabels]{enumitem}
\usepackage{xparse} 

\usepackage{graphicx}
\usepackage[dvipsnames]{xcolor}
\usepackage{tikz}
\usepackage{MnSymbol} 

\usepackage{todonotes}
\usepackage[
    deletedmarkup=sout,
    commentmarkup=uwave,
    authormarkuptext=name
]{changes}

\usepackage{comment}

\usepackage{thm-restate}
\usepackage{graphicx}
\usepackage[dvipsnames]{xcolor}
\usepackage{tikz}
\setlist[enumerate,1]{label=(\arabic*), ref=(\arabic*)}
\setlist[enumerate,3]{label=(\roman*), ref=(\roman*)}

\allsectionsfont{\boldmath} 

\titlelabel{\thetitle.\quad}

\theoremstyle{plain}
\newtheorem{theorem}{Theorem}[section]
\newtheorem{lemma}[theorem]{Lemma}
\newtheorem{corollary}[theorem]{Corollary}
\newtheorem{proposition}[theorem]{Proposition}
\newtheorem{observation}[theorem]{Observation}

\newtheorem{claim}{Claim}[theorem]
\newtheorem*{claim*}{Claim}

\makeatletter
\newenvironment{claimproof}[1][Proof]{\par
	\pushQED{\qed}%
	
	\normalfont \topsep6\p@\@plus6\p@\relax
	\trivlist
	\item[\hskip\labelsep
	\textit{#1}\@addpunct{.}~]\ignorespaces
}{%
	\popQED\endtrivlist\@endpefalse
}
\makeatother

\newlist{Cases}{enumerate}{3}
\setlist[Cases]{parsep=0pt plus 1pt}
\setlist[Cases,1]{wide=0pt, listparindent=\parindent,
    label = \textbf{Case~\arabic*:}, ref = \arabic*}
\setlist[Cases,2]{wide=\parindent, listparindent=\parindent,
    label = \textbf{Case~\arabic{Casesi}-\arabic{Casesii}:}}

\crefname{Casesi}{case}{cases}

\newcounter{case}
\AtBeginEnvironment{proof}{\setcounter{case}{0}}

\crefname{case}{case}{cases}

\theoremstyle{definition}
\newtheorem{definition}[theorem]{Definition}
\newtheorem{remark}[theorem]{Remark}

\newcommand{\calG}{\mathcal{G}}

\newcommand{\calP}{\mathcal{P}}

\newcommand{\ve}{\varepsilon}

\makeatletter
\NewDocumentCommand{\xsideset}{mmme{_^}}{%
  \mathop{%
    \settowidth{\dimen0}{$\m@th\displaystyle#3$}%
    \dimen0=.5\dimen0
    \settowidth{\dimen2}{$%
      \m@th\displaystyle#3%
      \IfValueT{#4}{_{#4}}%
      \IfValueT{#5}{^{#5}}%
    $}%
    \dimen2=.5\dimen2
    \advance\dimen2 -\dimen0
    \sbox6{\scriptspace\z@$\displaystyle{\vphantom{#3}}#1$}
    \sbox8{\scriptspace\z@$\displaystyle{\vphantom{#3}}#2$}
    \ifdim\wd6>\dimen2 \kern\dimexpr\wd6-\dimen2\relax\fi
    {%
     \mathop{\llap{\copy6}{\displaystyle#3}\rlap{\copy8}}\limits
     \IfValueT{#4}{_{#4}}%
     \IfValueT{#5}{^{#5}}%
    }%
    \ifdim\wd8>\dimen2 \kern\dimexpr\wd8-\dimen2\relax\fi
  }%
}
\makeatother

\let\originalleft\left
\let\originalright\right
\renewcommand{\left}{\mathopen{}\mathclose\bgroup\originalleft}
\renewcommand{\right}{\aftergroup\egroup\originalright}

\makeatletter
\newcases{lrcases*}
  {\quad}
  {$\m@th{##}$\hfil}
  {{##}\hfil}
  {\lbrace}
  {\rbrace}
\makeatother

\usetikzlibrary{matrix,arrows,decorations.pathmorphing}
\usetikzlibrary{positioning}
\usetikzlibrary{graphs} 
\usetikzlibrary{graphs.standard}

\title{On rainbow Tur\'{a}n Densities of Trees}
\author{Seonghyuk Im\thanks{Department of Mathematical Sciences, KAIST, South Korea Email:{\tt $\{$seonghyuk, jaehoon.kim, hyunwoo.lee, hss21$\}$@kaist.ac.kr}}~\thanks{Extremal Combinatorics and Probability Group (ECOPRO), Institute for Basic Science (IBS).}
\and Jaehoon Kim\footnotemark[1] \and Hyunwoo Lee\footnotemark[1]~\footnotemark[2]\and Haesong Seo\footnotemark[1]}

\usepackage[
    deletedmarkup=sout,
    commentmarkup=uwave,
    authormarkuptext=name
]{changes}

\begin{document}

\maketitle

\begin{abstract}

For a given collection $\calG = (G_1,\dots, G_k)$ of graphs on a common vertex set $V$, which we call a \emph{graph system}, a graph $H$ on a vertex set $V(H) \subseteq V$ is called a \emph{rainbow subgraph} of $\calG$ if there exists an injective function $\psi:E(H) \rightarrow [k]$ such that $e \in G_{\psi(e)}$ for each $e\in E(H)$. The maximum value of $\min_{i}\{|E(G_i)|\}$ over $n$-vertex graph systems $\calG$ having no rainbow subgraph isomorphic to $H$ is called the rainbow Tur\'{a}n number $\mathrm{ex}_k^{\ast}(n, H)$ of $H$. 

In this article, we study the rainbow Tur\'{a}n density $\pi_k^{\ast}(T) = \lim_{n \rightarrow \infty} \frac{\mathrm{ex}_k^{\ast}(n, T)}{\binom{n}{2}}$ of a tree $T$. 
While the classical Tur\'{a}n density $\pi(H)$ of a graph $H$ lies in the set $\{1-\frac{1}{t} : t\in \mathbb{N}\}$, the rainbow Tur\'{a}n density exhibits different behaviors as it can even be an irrational number.
Nevertheless, we conjecture that the rainbow Tur\'{a}n density is always an algebraic number.
We provide evidence for this conjecture by proving that the rainbow Tur\'{a}n density of a tree is an algebraic number.
To show this, we identify the structure of extremal graphs for rainbow trees.
Moreover, we further determine all tuples $(\alpha_1,\dots, \alpha_k)$ such that every graph system $(G_1,\dots,G_k)$ satisfying $|E(G_i)|>(\alpha_i+o(1))\binom{n}{2}$ contains all rainbow $k$-edge trees.
In the course of proving these results, we also develop the theory on the limit of graph systems.
\end{abstract}


\section{Introduction}\label{sec:intro}
Throughout this paper, every graph is simple and a multigraph refers to a multigraph with no loops unless we state otherwise.

A collection of graphs $\calG = (G_1, \dots, G_k)$ on a common vertex set $V$ is called a \emph{graph system} on $V$ with $k$ colors, and we call the number $k$ as the order of the graph system $\calG$. We denote by $V(\calG)$ the common vertex set $V$.
A multigraph $H$ satisfying $V(H) \subseteq V$ is called a \emph{rainbow subgraph} of $\calG$ if there exists an injective function $\psi:E(H) \rightarrow [k]$ such that $e \in G_{\psi(e)}$ holds for each $e\in E(H)$.

One of the central topics in extremal combinatorics is to determine the maximum number of edges of a graph on $n$ vertices that does not contain a copy of a given graph $H$ as a subgraph. 
Such a maximum is the \emph{extremal number} of $H$ on $n$ vertices and is denoted by $\mathrm{ex}(n, H)$. 
Tur\'{a}n~\cite{Turan41} (see also~\cite{Erdos70}) in 1941 first determined the extremal number of the complete graph $K_r$.
Erd\H{o}s, Stone and Simonovits~\cite{Erdos66, Erdos46} found a surprising connection between $\mathrm{ex}(n,H)$ and the chromatic number $\chi(H)$ of $H$ by proving that for every graph $H$, its extremal number is $\mathrm{ex}(n, H) = \left(1 - \frac{1}{\chi(H)-1} + o(1) \right)\binom{n}{2}$.
Since then, the extremal numbers and their variations have been broadly studied, providing elaborate theories and a deep understanding of the relations between various global statistics and local structures of graphs. 
See~\cite{Furedi2013} for a general survey. 


One very popular variation of the extremal number is the rainbow generalization.
In other words, we want to determine the ``largest'' $n$-vertex graph systems that do not contain a copy of a multigraph $H$ as a rainbow subgraph.
Depending on the contexts, different measures for the ``size'' have been introduced: 
the sum $\sum_{i} |E(G_i)|$~\cite{Chakraborti22, Keevash04}, the minimum $\min_{i} |E(G_i)|$~\cite{Aharoni20, Babinski22}, the product $\prod_{i} |E(G_i)|$~\cite{Frankl22, Frankl22a}, the minimum of the spectral radius $\min_i \lambda_1(G_i)$~\cite{Guo22} and other measures \cite {falgasravry2023a, falgasravry2023}.
Surprisingly, this problem highly depends on the choice of measures.

To maximize the sum $\sum_{i=1}^k |E(G_i)|$ of the number of edges of a graph system without a rainbow copy of $H$, there are two following constructions for natural lower bounds.
One is that all the $G_i$'s are the same Tur\'an graph; the other one is that there are $|E(H)|-1$ copies of the complete graph and all the other graphs are empty.
Keevash, Saks, Sudakov and Verstra\"{e}te~\cite{Keevash04} proved that one of the previous two constructions is extremal (i.e., a maximizer) when $H = K_r$, generalizing the Tur\'{a}n theorem.
They also proved that if $H$ is $3$-color-critical, i.e., $\chi(H)=3$ and $H$ has an edge $e$ with $\chi(H-e)=2$, then the same conclusion holds. 
This extends the result of Simonovits~\cite{Simonovits68} stating that the Tur\'{a}n graph is extremal when $H$ is $k$-color-critical for some $k \geq 3$.
Recently, Chakraborti, Liu, Seo, the second and the third authors~\cite{Chakraborti22} generalized the previous result to every $4$-color-critical graph $H$ and almost all $k$-color-critical graphs $H$ for $k \geq 5$.
Furthermore, they proved that if $H$ is not $k$-color-critical for some $k \geq 3$ and $n$ is sufficiently large, both of the constructions are not extremal.

By enforcing the graphs in $\mathcal{G}$ to be pairwise edge-disjoint matchings and letting $k$ be large, the maximum of $\sum_{i=1}^k |E(G_i)|$ is the maximum number of edges in a properly edge-colored graph with no rainbow subgraph $H$.
Such a variation has been extensively studied in~\cite{alon2023, Das2013, Ergemlidze2019, Janzer2023, janzer2022, Keevash2007, kim2022, tomon2022}

Regarding the minimum $\min_{i} |E(G_i)|$, define the \emph{rainbow extremal number} $\mathrm{ex}_k^{\ast}(n, H)$ of a multigraph $H$ to be the maximum of $\min_{1 \leq i \leq k} |E(G_i)|$ among all rainbow $H$-free graph systems $\calG=(G_1, \ldots, G_k)$ on $n$ vertices. 
Aharoni, Devos, de la Maza, Montejano and \v{S}\'{a}mal~\cite{Aharoni20} proved that $\mathrm{ex}_3^{\ast}(n, K_3) = \lfloor \frac{26-2\sqrt{7}}{81}n^2\rfloor $.
The fact that $\frac{26-2\sqrt{7}}{81}$ is an irrational number larger than the classical Tur\'an density $\frac{1}{4}$
signifies the difference between the rainbow setting and the classical setting, and renders the rainbow extremal number more intriguing.

In general cases, the problem gets more complicated and there are only few known results regarding the rainbow extremal number $\mathrm{ex}_k^{\ast}(n,H)$. 
As a corollary of Keevash, Saks, Sudakov and Verstra\"{e}te~\cite{Keevash04}, we have $\mathrm{ex}_k^{\ast}(n, K_3) \leq \frac{n^2}{4}+o(n^2)$ when $k \geq 4$. 
Babi\'{n}ski and Grzesik~\cite{Babinski22} recently determined the value of the rainbow extremal number $\mathrm{ex}_k^{\ast}(n, P_3)$ of the $3$-edge path up to additive $o(n^2)$ error term for every $k \geq 3$.
However, many cases including the complete graph on at least four vertices are still wide open.

\subsection{Main results}\label{subsec:1.1}
The \emph{rainbow Tur\'{a}n density} of a multigraph $H$ is defined as follows.
\begin{align}\label{eq: def rainbow Turan density}
    \pi_k^{\ast}(H) := \lim_{n \rightarrow \infty} \frac{\mathrm{ex}_k^{\ast}(n, H)}{{n \choose 2}}.
\end{align}
The first natural question is whether the limit exists.
Unlike the case of maximizing $\sum_{i=1}^k |E(G_i)|$, the standard proof (cf. the survey of Keevash~\cite{Keevash11}) does not directly work.
So we first prove the existence of this limit.
\begin{theorem}\label{thm:existence}
    For every multigraph $H$ and $k \geq |E(H)|$, the limit $\displaystyle \pi^{\ast}_k(H)=\lim_{n \rightarrow \infty} \frac{\mathrm{ex}_k^{\ast}(n, H)}{{n \choose 2}}$ 
    exists.
\end{theorem}

The classical Tur\'{a}n density $\pi(H)= \lim_{n \rightarrow \infty} \frac{\mathrm{ex}(n, H)}{{n \choose 2}}$ is always in the set $\{1-\frac{1}{t}:t\in \mathbb{N}\}$. This is no longer true for the rainbow Tur\'{a}n density as $\pi_{3}^{*}(K_3) = \frac{26-2\sqrt{7}}{81}$.
On the other hand, such an irrational number $\frac{26-2\sqrt{7}}{81}$ was obtained in~\cite{Aharoni20} from a polynomial optimization problem regarding the size of certain vertex sets.
Based on this example, it is reasonable to conjecture that all rainbow Tur\'{a}n densities $\pi_{k}^{*}(H)$ can be attained from such an optimization problem, and they are algebraic numbers.
We prove this conjecture for all trees $T$ by analyzing the structure of extremal graph systems.
Roughly speaking, for a given tree $T$, the vertex set $V(\calG)$ of an extremal graph system $\calG=(G_1, \ldots, G_k)$ can be partitioned into a finite number of parts, where each $G_i$ is a union of cliques and complete bipartite graphs on and between those parts.
In addition, the number of parts only depends on the tree $T$.
See Theorem~\ref{thm:structure} for the full statement.
Hence the rainbow Tur\'{a}n density $\pi^{\ast}_k(T)$ is a solution to a polynomial optimization problem.
As a consequence, it can be proven to be an algebraic number using basic tools from real algebraic geometry.
In summary, we prove the following.
\begin{theorem}\label{thm:algebraic_number}
    For a tree $T$, the rainbow Tur\'{a}n density $\pi_k^{\ast}(T)$ can be computed by solving finitely many polynomial optimization problems.
    Moreover, $\pi_k^{\ast}(T)$ is an algebraic number.
\end{theorem}
Such a polynomial optimization may have exponentially many variables in terms of $|T|$, so it is not practical to compute $\pi_k^{\ast}(T)$ via this optimization even if the tree $T$ is small. 
Nevertheless, we use the structure of extremal graph systems to verify that the rainbow Tur\'{a}n density of a tree $T$ is maximized when $T$ is a star.
\begin{theorem}\label{thm:tree_less_star}
    If $T$ is a $k$-edge tree, then $\pi_k^{\ast}(T) \leq \pi_k^{\ast}(K_{1, k}) = \left(\frac{k-1}{k}\right)^2$.
\end{theorem}
We further prove that one can completely determine the tuples $(\alpha_1, \ldots, \alpha_k) \in [0, 1]^k$ such that for every $\varepsilon>0$, a graph system $\mathcal{G}$ on $n$ vertices has a rainbow copy of all $k$-edge trees whenever $\frac{1}{n}\ll \varepsilon$ and $|E(G_i)| > (\alpha_i+\varepsilon) \binom{n}{2}$ holds for all $i\in [k]$.
Note that the determination of all such tuples were questioned in~\cite{Aharoni20} for the case of rainbow triangles.

\begin{theorem}\label{thm:tree_less_star_2}
    Let $\alpha_1, \ldots, \alpha_k$ be nonnegative reals such that $\sum_{i=1}^k (1-\sqrt{\alpha_i}) < 1$.
    For every $\varepsilon>0$, there exists an integer $n_0$ such that for every $n \geq n_0$ and every $k$-edge tree $T$, a graph system $\mathcal{G}=(G_1, \ldots, G_k)$ on $n$ vertices with $|E(G_i)| \geq (\alpha_i + \varepsilon){n \choose 2}$ contains a rainbow copy of $T$.
\end{theorem}
We note that the condition $\sum_{i \in [k]} (1-\sqrt{\alpha_i}) < 1$ is tight.
Let $\alpha_1, \ldots, \alpha_k$ be nonnegative reals such that $\sum_{i\in [k]} (1-\sqrt{\alpha_i}) \geq 1$.
One can choose subsets $V_1, \dots, V_k \subseteq [n]$ such that $|V_i| = \lceil (1-\sqrt{\alpha_i})n \rceil$ and $\bigcup_{i \in [k]} V_i = [n]$.
Let $G_i^{n}$ be a graph on $[n]$ defined by a complete graph on $[n] \setminus V_i$.
Then the graph system $\mathcal{G}^{n} = (G_1^{n}, \ldots, G_k^{n})$ does not contain a rainbow copy of $K_{1, k}$ and $|E(G_i^{n})| \cdot {n \choose 2}^{-1} \rightarrow \alpha_i$ as $n$ tends to infinity.

Finally, we show the following three basic properties of the rainbow Tur\'{a}n density.
\begin{proposition}\label{thm:bipartite_max}
    For every bipartite multigraph $H$ and $k \geq |E(H)|=\ell$, we have $\pi_k^{\ast}(H) \leq \frac{\ell-1}{k}.$
\end{proposition}
For a given multigraph $H$, denote by $\mathrm{sim}(H)$ the graph obtained by replacing each multiple edge with a single edge.
We call it the \emph{simplification} of $H$.

\begin{proposition}\label{thm:limit_of_turan_density}
    For every multigraph $H$, we have $\lim_{k \rightarrow \infty} \pi_k^{\ast}(H) = 1-\frac{1}{\chi(\mathrm{sim}(H))-1}.$ In other words, $\pi_k^{\ast}(H)$ converges to the Tur\'{a}n density of its simplification as $k$ tends to infinity.
\end{proposition}

\begin{proposition}[Supersaturation]\label{thm:supersaturation}
    For every multigraph $H$ and every $\varepsilon>0$, there exists $\delta>0$ such that if an $n$-vertex graph system $\mathcal{G}$ of order $k$ has at least $(\pi_k^{\ast}(H) + \varepsilon){n \choose 2}$ edges, then it has at least $\delta n^{|V(H)|}$ rainbow copies of $H$.
\end{proposition}

\subsection{The limit of graph systems}
To prove our main theorems, we develop the theory on the limit of graph systems.
For two measurable functions $f, g:X \rightarrow \mathbb{R}$ from a measure space $X$, we write $f \equiv g$ if $f=g$ holds almost everywhere.
For a measurable subset $A \subset X$, we say $f \equiv g$ on $A$ if $f = g$ holds almost everywhere on $A$.
We use the standard Lebesgue measure on $\mathbb{R}^n$.
A symmetric measurable function $W:[0, 1]^2 \rightarrow [0, 1]$ is called a \emph{graphon}. 
Let $\mathcal{W}_0$ be the space of all graphons.
An $n$-vertex graph $G$ can be interpreted as a graphon by partitioning $[0, 1]$ into $n$ intervals $I_1=[0, \frac{1}{n}), I_2=[\frac{1}{n}, \frac{2}{n}), \ldots, I_n=[\frac{n-1}{n}, 1]$ and by setting its value to be constantly $1$ on $I_i \times I_j$ if $ij \in E(G)$ and constantly $0$ otherwise.
It is well-known that a graphon can be interpreted as the limit of a sequence of graphs.
Furthermore, there is a metric called the ``cut-distance'' between graphons.
By identifying the graphons with cut-distance zero, we obtain a compact metric space $\widetilde{\mathcal{W}}_0$.
Also, we have a family of continuous functions, called the ``subgraph density'' functions, which determines the metric topology of $\widetilde{\mathcal{W}}_0$.
We further explain the theory of graphons and related definitions in Section~\ref{subsec:graphon_def}.

One may expect that a natural generalization of graph systems of order $k$ is the set of $k$-tuples of graphons $\mathbf{W} = (W_1, \ldots, W_k) \in \mathcal{W}_0^{k}$. 
However, this definition is not appropriate for extending theorems of graphon to the graph system setting.
Consider two graph systems $\mathcal{G}_1^n = (G_1^n, G_2^n)$ and $\mathcal{G}_2^n = (G_3^n, G_4^n)$, where $G_1^n, G_3^n, G_4^n$ are mutually independent binomial random graphs $G(n, 1/2)$ and $G_2^n$ is the complement of $G_1^n$.
Let $H$ be a multigraph with two vertices and two parallel edges.
Then $\mathcal{G}_1^n$ has no rainbow copies of $H$, while $\mathcal{G}_2^n$ contains ${n \choose 2}/4$ rainbow copies of $H$ in expectation.
On the other hand, both $\mathcal{G}_1^n$ and $\mathcal{G}_2^n$ component-wise converge to $\mathbf{W} = (\frac{1}{2} \mathbbm{1}_{[0, 1]^2}, \frac{1}{2} \mathbbm{1}_{[0, 1]^2})$ with high probability. 
It follows that the ``rainbow $H$-density'' is not a continuous function on the product space $\mathcal{W}_0^{k}$.
Even if we avoid using multigraphs, the same problem appears when we count the number of, for instance, induced copies of $K_{1, 2}$. 
In this fashion, we also need to encode the information regarding the intersection of graphs as follows: 
\begin{definition}
    A \emph{graphon system of order $k$} is a tuple of graphons $\mathbf{W} = (W_I)_{I \subseteq [k]}$ with $W_\emptyset \equiv 1$.
    When $I = \{ i \} \subseteq [k]$ is a singleton, we simply write $W_i = W_{\{ i \}}$.
    For graphons $W_1, \ldots, W_k$, let $\mathrm{span}(W_1, \ldots, W_k)$ be the classical graphon system $\mathbf{W}$ of order $k$ defined by $W_I = \prod_{i \in I} W_i$.
\end{definition}
As we view a graph as a graphon, a graph system $\calG = (G_1, \ldots, G_k)$ can be regarded as a graphon system by letting $G_I = \bigcap_{i \in I} G_i$.
Especially, as a graph is a graphon taking values only in $\{ 0,1 \}$, the graphon system $\mathrm{span}(G_1, \ldots, G_k)$ corresponds to the graph system $\mathcal{G}$. 
Throughout this paper, we identify a graph system $\mathcal{G}$ with the corresponding graphon system $\mathrm{span}(G_1, \ldots, G_k)$ when there is no confusion.  
In the previous example, $\mathcal{G}_1^n$ converges to the graphon system $\mathbf{W}$ with ${W}_{1}, {W}_{2} \equiv \frac{1}{2}$ and $W_{\{1, 2\}} \equiv 0$, whereas $\mathcal{G}_2^n$ converges to the graphon system $\mathbf{W}$ with ${W}_{1}, {W}_{2} \equiv \frac{1}{2}$ and ${W}_{\{1,2\}} \equiv \frac{1}{4}$.
In \Cref{subsec:graphon_system_def} and \Cref{sec:induced_density}, we discuss the definition of homomorphism density and induced homomorphism density of a graphon system.



Graphon systems form a compact metric space with respect to a suitable metric, called the \emph{cut-distance} (see \Cref{thm:compact}).
The metric topology is determined by a family of continuous functions, called the \emph{(rainbow) homomorphism density} (see \Cref{thm:counting_lemma} and \Cref{thm:inverse_counting}).
However, the set of graph systems is not dense in this space.
For example, if $W_{\{ 1,2 \}} > W_{ 1 }$ on a set of positive measures in $[0,1]^2$, then we will never be able to find a sequence of graph systems $(G_1^n, G_2^n)$ whose limit is $(W_{\emptyset}, W_{ 1 },W_{ 2 },W_{\{ 1,2 \}})$ because we always have $G_{\{ 1,2 \}}^n \leq G_1^n$.
The next definition provides the complete characterization of graphon systems that are limits of graph systems.

\begin{definition}\label{def: barW}
    Let $\mathbf{W}=(W_I)_{I \subseteq [k]}$ be a graphon system of order $k$.    
    For each $I \subseteq [k]$, we define $\overline{W}_I$ inductively by:
    \begin{itemize}
        \item $\overline{W}_{[k]} = W_{[k]}$,
        \item $\overline{W}_I = W_I - \sum_{J \supsetneq I} \overline{W}_J$ if $I \neq [k]$.
    \end{itemize}
    The graphon system $\mathbf{W}$ is called an \emph{admissible graphon system} if $\overline{W}_I(x, y) \geq 0$ for almost every $(x,y)\in [0,1]^2$ and for any $I \subseteq [k]$. 
    A graphon system $\mathbf{W}=(W_I)_{I \subseteq [k]}$ is \emph{classical} if $W_I = \prod_{i \in I} W_{\{i\}}$ holds almost everywhere for each $I\subseteq [k]$.
\end{definition}
Observe that a classical graphon system $\mathbf{W} = (W_I)_{I \subseteq [k]}$ satisfies $\overline{W}_I = \prod_{i \in I}W_i \prod_{i \notin I} (1-W_i) \geq 0$
for all $I\subseteq [k]$, so a classical graphon system is an example of admissible graphon systems.
However, the space of admissible graphon systems is strictly larger than the space of classical graphon systems, and the set of classical graphon systems is not compact.
The previous example $\mathcal{G}_1^n = (G_1^n, G_2^n)$ gives a sequence of classical graphon systems whose limit is not classical.
In the following theorem, $\widetilde{\mathcal{W}}_0^{(k)}$ is the space of graphon systems of order $k$ equipped with the metric $\delta_{\square}$, cf. Section~\ref{subsec:graphon_system_def}.

\begin{theorem}\label{thm:classical}
    The space of admissible graphon systems is the closure of the set of graph systems in $(\widetilde{\mathcal{W}}_0^{(k)}, \delta_{\square})$. 
    In other words, every sequence of graph systems has a subsequence convergent to an admissible graphon system and every admissible graphon system is a limit of a sequence of graph systems.
\end{theorem}


Finally, we note some benefits for developing the theory on the limit of graph systems.
First, the computation in the proof of \Cref{thm:tree_less_star_2} becomes clearer with this theory. 
Second, as we note in Remark~\ref{rmk1}, the limit theory allows us to easily generalize Theorem~\ref{thm:existence} and the results on the structure of extremal graphon systems to the various settings including:
\begin{itemize}
    \item maximizing the product $\frac{\prod_{i=1}^k |E(G_i)|}{{n \choose 2}^k}$;
    \item maximizing $\min_{1 \leq i \leq k} \frac{\lambda_1(G_i)}{n}$; and
    \item generalized Tur\'{a}n problems: maximizing the rainbow density of $H_1$ in rainbow $H_2$-free graph systems.
\end{itemize}

\textbf{Organization.}
In Section~\ref{sec:prelim}, we introduce the theory of graphons and generalize it to the graphon systems.
In Section~\ref{sec:graphon_system}, we prove Theorem~\ref{thm:classical} determining the closure of the set of graph systems.
In Section~\ref{sec:applications}, we prove several properties of rainbow Tur\'{a}n density, including \Cref{thm:existence} and \Cref{thm:bipartite_max}--\ref{thm:supersaturation}. 
In Section~\ref{sec:structure} and~\ref{sec:proof-tree-star}, we investigate the structure of extremal graphon systems having no rainbow trees and prove \Cref{thm:algebraic_number} and~\ref{thm:tree_less_star}, respectively.
In Appendix~\ref{appendix:easy_proofs}$-$\ref{sec:induced_density}, we collect extensions of theorems on a graphon to a graphon system which can be easily proved by similar proof.


\section{Preliminaries and notations}\label{sec:prelim}
Throughout this paper, we assume that all the subsets of $\mathbb{R}^n$ and the functions between them that we deal with are measurable.
\subsection{Colored subgraphs}
In this section, we define some notions for edge-colored graphs.
We also define homomorphisms between two multigraphs that preserve pre-colorings, and extend the rainbow Tur\'{a}n problem to a more general pre-colored Tur\'{a}n problem.

\begin{definition}
    Let $H$ be a multigraph.
    For each $uv\in \binom{V(H)}{2}$, denote by $E(H)_{uv}$ the set of edges of $H$ between $u$ and $v$. In particular, $E(H)_{uv} = \emptyset$ when $u$ and $v$ are not adjacent in $H$.
    If $H$ is clear in the context, we simply write $E_{uv}$.
    
     For a set $X\subseteq E(H)$ of edges, a map $\psi: X\rightarrow [k]$ is called a \emph{pre-coloring} of $H$ if $\psi$ is injective on $E_{uv} \cap X$ for each $uv\in \binom{V(H)}{2}$. Denote by $\textbf{dom}(\psi)=X$ the domain of $\psi$.
    If $\psi$ is injective on its domian, $\psi$ is called a \emph{rainbow pre-coloring} of $H$. If $\textbf{dom}(\psi)=E(H)$, we call $\psi$ a \emph{coloring} of $H$.
    A tuple $(H, \psi)$ of a multigraph and its  pre-coloring is called a \emph{pre-coloring tuple}. If $\psi$ is a rainbow pre-coloring, $(H, \psi)$ is called a \emph{rainbow pre-coloring tuple}; if $\psi$ is a (resp. rainbow) coloring, $(H,\psi)$ is called a \emph{(resp. rainbow) coloring tuple}.
\end{definition}

As above, a pre-coloring tuple is a multigraph some of whose edges are already colored. 
Using the definition below, we can specify a certain subgraph of $G$ whose edges are colored consistently with a given pre-coloring tuple $(H,\psi)$.

\begin{definition}
For multigraphs $G$ and $H$, a \emph{multigraph homomorphism} $(f, \widetilde{f}):H \rightarrow G$ is a tuple of a graph homomorphism between vertex sets $f:V(H) \rightarrow V(G)$ and a map $\widetilde{f}:E(H) \rightarrow E(G)$ such that $\widetilde{f}(E(H)_{uv})\subseteq E(G)_{f(u)f(v)}$ for all $uv\in \binom{V(H)}{2}$.
A multigraph homomorphism is an \emph{edge-preserving homomorphism} if $\widetilde{f}$ is injective. If there exists a multigraph homomorphism $(f,\widetilde{f})$ with bijective $\widetilde{f}$, the graph $G$ is called an \emph{edge-preserved homomorphic image} of $H$.


For two pre-coloring tuples $(H, \psi_1)$ and $(G, \psi_2)$, a \emph{color homomorphism} is an edge-preserving homomorphism $(f,\widetilde{f}): H \rightarrow G$ such that $\textbf{dom}(\psi_2) = \widetilde{f}(\textbf{dom}(\psi_1))$ and $\psi_1(e) = \psi_2(\widetilde{f}(e))$ for every $e \in \textbf{dom}(\psi_1)$.
A tuple $(G, \psi_2)$ is called a \emph{color homomorphic image} of $(H, \psi_1)$ if there exists a color homomorphism $(f, \widetilde{f})$ between them with $\widetilde{f}$ bijective.
\end{definition}

Recall the definition of the rainbow Tur\'{a}n density in \eqref{eq: def rainbow Turan density}.
We wish to extend this definition to pre-coloring tuples.  

\begin{definition}
    For a pre-coloring tuple (resp., rainbow pre-coloring tuple) $(H, \psi)$, a graph system $\mathcal{G}=(G_1, \ldots, G_k)$ on $V$ has $(H, \psi)$ as a \emph{colored subgraph} (resp., \emph{rainbow subgraph}) if there exist a multigraph $H'$ on the vertex set $V(H')\subseteq V$ together with a multigraph isomorphism $f:H \rightarrow H'$ and a (resp., injective) function $\psi': E(H') \to [k]$ such that  $e\in E(G_{\psi'(e)})$ for each $e\in E(H')$ and $\psi(e) = \psi'(f(e))$ for all $e\in \textbf{dom}(\psi)$.
    
    For a family $\mathcal{F}$ of rainbow pre-coloring tuples,
    $\calG$ is said to be \emph{rainbow $\mathcal{F}$-free} if it does not have $(H, \psi)$ as a rainbow subgraph for every $(H, \psi) \in \mathcal{F}$.
    When $\mathcal{F}$ consists of only one element $(H,\psi)$, we say that $\calG$ is \emph{rainbow $(H, \psi)$-free}.
    If $\textbf{dom}(\psi)=\emptyset$, the graph system $\calG$ is said to be \textit{rainbow $H$-free}.
\end{definition}

\begin{definition}
    Let $\mathcal{F}$ be a family of rainbow pre-coloring tuples. We define the \emph{rainbow extremal number} of $\mathcal{F}$ by
    $$\mathrm{ex}_k^{\ast}(n, \mathcal{F}) : = \max \left\{ \min_{1 \leq i \leq k} |E(G_i)| \,\middle\mid\, (G_1, \ldots, G_k)\text{ is a rainbow $\mathcal{F}$-free graph system on } [n] \right\}.$$
    The \emph{rainbow Tur\'{a}n density} of $\mathcal{F}$ is defined by
     $$\pi_k^{\ast}(\mathcal{F}) = \lim_{n \rightarrow \infty} \frac{\mathrm{ex}_k^{\ast}(n, \mathcal{F})}{{n \choose 2}}.$$
    When $\mathcal{F}=\{(H, \psi)\}$, then we write $\mathrm{ex}_k^{\ast}(n, H, \psi)$ and $\pi_k^{\ast}(H, \psi)$, respectively. 
    We often omit $\psi$ when $\textbf{dom}(\psi) = \emptyset$.
\end{definition}

We will prove the existence of $\pi_k^{\ast}(H, \psi)$ for every pre-coloring tuple $(H,\psi)$ in Theorem~\ref{thm:stronger_existence}.

\subsection{Preliminaries on graphons}\label{subsec:graphon_def}
In this section, we describe the theory on graphons developed by Lov\'asz~\cite{Lovasz2012}.

A \emph{graphon} is a symmetric function $W:[0, 1]^2 \rightarrow [0, 1]$.
Denote by $\mathcal{W}_0$ the set of graphons in $L^2([0, 1]^2)$.
For a measure-preserving map $\varphi:[0, 1] \rightarrow [0, 1]$ and a graphon $W$, define $W^{\varphi}(x, y) := W(\varphi(x), \varphi(y))$.
For a bounded symmetric function $W:[0, 1]^2 \rightarrow \mathbb{R}$, its \emph{cut-norm} is defined by
$$\lVert W \rVert_{\square} := \sup_{S, T \subseteq [0, 1]} \left| \int_{S \times T} W \right|. $$
The \emph{cut-distance} $\delta_{\square}(W, U)$ between two graphons $W$ and $U$ is defined by
$$\delta_{\square}(W, U) = \inf_{\varphi} \lVert W - U^{\varphi} \rVert_{\square},$$
where the infimum is taken over all measure-preserving invertible maps $\varphi:[0, 1] \rightarrow [0, 1]$.
Let $\widetilde{\mathcal{W}}_0$ be the quotient space of $\mathcal{W}_0$ taken by identifying graphons $U$ with $W$ whenever $\delta_\square(U, W)=0$.
\begin{theorem}[\cite{Lovasz2012}, Theorem 9.23]
    $(\widetilde{\mathcal{W}}_0, \delta_\square)$ is a compact metric space.
\end{theorem}

After taking quotient, the following homomorphism density that generalizes the number of copies of a graph becomes a continuous function on the space $(\widetilde{\mathcal{W}}_0, \delta_\square)$.
\begin{definition}\label{def: hom density}
    For a graph $H$, the \emph{homomorphism density} of $H$ is defined by 
    $$t_H(W) = \int_{[0, 1]^{|V(H)|}} \prod_{e=uv \in E(H)} W(x_u, x_v) \prod_{v \in V(H)} dx_v.$$
\end{definition}
\begin{theorem}[\cite{Lovasz2012}, Counting Lemma]
    For graphons $W$ and $U$, we have
    $$|t_{H}(W) - t_H(U)| \leq |E(H)| \delta_{\square}(U, W).$$
    In particular, $t_H$ is a well-defined continuous function on $(\widetilde{\mathcal{W}}_0, \delta_{\square})$.
\end{theorem}
Moreover, the following generalization of the above theorem will be used later when we discuss the graphon systems.
\begin{theorem}[\cite{Lovasz2012}, Counting Lemma for decorated graphs]\label{thm:counting_lemma_for_decorated}
    For a simple graph $H$ and two tuples of graphons $(W_{uv})_{uv \in E(H)}$ and $(U_{uv})_{uv \in E(H)}$, we have 
    $$\left|\int_{[0, 1]^{|V(H)|}} \left( \prod_{uv \in E(H)} W_{uv}(x_u, x_v)  - \prod_{uv \in E(H)} U_{uv}(x_u, x_v) \right) \prod_{v \in V(H)} dx_v\right| \leq \sum_{u  \in E(H)} \lVert W_{uv} - U_{uv} \rVert_\square.$$
\end{theorem}

\subsection{Graphon systems}\label{subsec:graphon_system_def}
We now generalize the previous concepts to graphon systems.
We introduce the cut-distance between two graphon systems.

\begin{definition}\label{def:cut_norm}
    For a graphon system $\mathbf{W}=(W_I)_{I \subseteq [k]}$, set $\mathbf{W}^\varphi := (W_I^\varphi)_{I \subseteq [k]}$ for a measure-preserving map $\varphi:[0,1] \rightarrow [0,1]$.
    The \emph{cut-norm} of a tuple $\mathbf{W}=(W_1, \ldots, W_s)$ of bounded symmetric functions on $[0,1]^2$ is defined by 
    $$\lVert \mathbf{W}\rVert_{\square} = \sup_{S, T \subseteq [0, 1]}  \sum_{i=1}^s \left| \int_{S \times T} W_i~ dx dy\right|.$$
    For two graphon systems $\mathbf{W}=(W_I)_{I \subseteq [k]}$ and $\mathbf{U} = (U_I)_{I \subseteq [k]}$, let $d_{\square}(\mathbf{W}, \mathbf{U}) = \lVert \mathbf{W}-\mathbf{U} \rVert_\square$.
    The \emph{cut-distance} is defined by 
    $$\delta_{\square}(\mathbf{W}, \mathbf{U}) = \inf_{\varphi} d_{\square} (\mathbf{W}, \mathbf{U}^{\varphi})
    =  \inf_{\varphi} \sup_{S, T \subseteq [0, 1] }  \sum_{I \subseteq [k]} \left| \int_{S \times T} W_I - U_I^{\varphi} \right|,
    $$
    where the infimum is taken over all measure-preserving invertible maps $\varphi:[0,1] \rightarrow [0,1]$.
    Let $\widetilde{\mathcal{W}}_0^{(k)}$ be the quotient space of graphon systems of order $k$ taken by identifying $\mathbf{W}$ with $\mathbf{U}$ whenever $\delta_{\square}(\mathbf{W}, \mathbf{U})=0$.
\end{definition}

Note that $\widetilde{\mathcal{W}}_0^{(k)}$ is different from the product of $2^k-1$ copies of $\widetilde{\mathcal{W}}_0$ because the same measure-preserving bijection $\varphi$ is applied to all the $U_i$ simultaneously.
It is easy to see that $\frac{1}{2^k} \sum_{I \subseteq [k]} \lVert W_I \rVert_{\square} \leq \lVert \mathbf{W} \rVert_{\square} \leq  \sum_{I \subseteq [k]} \lVert W_I \rVert_{\square}$ holds for any tuple $\mathbf{W}=(W_I)_{I \subseteq [k]}$ of bounded symmetric functions.
Hence our definition of cut-norm may be replaced by $\sum_{I \subseteq [k]} \lVert W_I \rVert_{\square}$.
However, the cut-distance between graphon systems cannot be bounded by the sum of the cut-distances of components.

We now define the colored homomorphism density and the rainbow homomorphism density of pre-coloring tuples.
\begin{definition}
    For a (resp., rainbow) coloring tuple $(H, \psi)$, the \emph{(resp., rainbow) $(H, \psi)$-density} of a graphon system $\mathbf{W}=(W_I)_{I \subseteq [k]}$ is defined by 
    $$ t^{*}_{(H, \psi)}(\mathbf{W}) = \int_{[0,1]^{|V(H)|}} \prod_{uv \in \binom{V(H)}{2}} W_{\psi(E_{uv})}(x_u, x_v) \prod_{v \in V(H)} dx_v.   $$ 
    For a rainbow pre-coloring tuple $(H, \psi)$ with $\mathbf{dom}(\psi)\subsetneq E(H)$, we define
    $$ t^{*}_{(H, \psi)}(\mathbf{W}) = \sum_{\widetilde{\psi}} t^{*}_{(H, \widetilde{\psi})}(\mathbf{W}), $$ 
    where the sum is taken over all injective functions $\widetilde{\psi}:E(H) \rightarrow [k]$ extending $\psi$. 
    When $\mathbf{dom}(\psi) = \emptyset$, we simply write $t^{*}_H(\mathbf{W})$.
\end{definition}

We are now ready to state the following theorems.
As the proof of Theorem~\ref{thm:compact} is similar to the graphon case, we postpone it to Appendix~\ref{appendix:easy_proofs}.


\begin{theorem}\label{thm:compact}
    The space $(\widetilde{\mathcal{W}}_0^{(k)}, \delta_{\square})$ is a compact metric space.
\end{theorem}
\begin{theorem}[Counting Lemma]\label{thm:counting_lemma}
    For two graphon systems $\mathbf{W}=(W_I)_{I\subseteq [k]}$ and $\mathbf{U}=(U_I)_{I\subseteq [k]}$ of order $k$ and a rainbow pre-coloring tuple $(H,\psi)$,
    $$|t^*_{(H, \psi)}(\mathbf{W}) - t^*_{(H, \psi)}(\mathbf{U})| \leq |E(H)| \delta_{\square}(\mathbf{W}, \mathbf{U}).$$
    In particular, $t_{(H, \psi)}$ is a well-defined continuous function on $(\widetilde{\mathcal{W}}_0^{(k)}, \delta_{\square})$.
\end{theorem}
\begin{proof}
    For any measure-preserving invertible map $\varphi:[0, 1] \rightarrow [0, 1]$, we have $t^*_{(H, \psi)}(\mathbf{U}) = t^*_{(H, \psi)}(\mathbf{U}^{\varphi})$.
    By applying Theorem~\ref{thm:counting_lemma_for_decorated} to $(W_{\varphi(uv)})_{uv\in E(H)}$ and $(U_{\varphi(uv)})_{uv\in E(H)}$ and taking the infimum over all $\varphi$, we get the result.
\end{proof}

\subsection{Real algebraic geometry}\label{appendix:real_algebraic_geometry}
In this section, we state a theorem which will be useful for proving \Cref{thm:algebraic_number}.

\begin{definition}{\cite[Theorem~1.2.2]{BCR1998}}
    A \textit{real closed field} is a field $K$ satisfying the following equivalent conditions:
    \begin{enumerate}
        \item $K$ can be ordered, and there is no nontrivial algebraic extension that extends an order;
        \item $K$ is not algebraically closed, but $K[\sqrt{-1}]$ is algebraically closed;
        \item $K$ has a unique ordering such that every nonnegative element is the square and every polynomial of odd degree over $K$ has a root in $K$.
    \end{enumerate}
\end{definition}

The field $\mathbb{R}$ of real numbers and the field $\mathbb{R}_{\mathrm{alg}} = \overline{\mathbb{Q}} \cap \mathbb{R}$ of real algebraic numbers are examples of real closed fields. There is a meta-theorem for real closed fields analogous to the Lefschetz principle for algebraically closed fields of characteristic 0.

\begin{theorem}[\cite{BCR1998}, Proposition~5.2.3, Transfer Principle]\label{thm:transfer_principle}
    Let $K$ be a real closed extension of a real closed field $F$, i.e., a field extension that is real closed. Let $\Phi$ be a formula without any free variable, written with a finite number of conjunctions, disjunctions, negations and existential quantifiers on variables, where atomic formulas are formulas of the kind $f(x_1, \ldots, x_n) \leq 0$ for some polynomial $f$ over $F$. Then $\Phi$ holds in $F$ if and only if it holds in $K$.
\end{theorem}



\section{Admissible graphon systems}\label{sec:graphon_system}
In this section, we prove \Cref{thm:classical}.
For a finite partition $\mathcal{P} = (I_1, \ldots, I_t)$ of $[0, 1]$ and a graphon $W$, let $W_{\mathcal{P}}$ be a step graphon defined by
$$W_{\calP}(x_0, y_0) = \frac{1}{\mu(I_{x_0})\mu(I_{y_0})} \int_{I_{x_0} \times I_{y_0}} W(x,y) dxdy, $$
if $\mu(I_{x_0})$, $\mu(I_{y_0})$ are both nonzero, where $\mu$ is the standard Lebesgue measure and $I_{x_0}$ (resp. $I_{y_0}$) is the part $I_i$ containing $x_0$ (resp. $y_0$);
if the measure of $I_{x_0}$ or $I_{y_0}$ is zero, we take $W_{\calP}(x_0,y_0)=0$.
For a graphon system $\mathbf{W} = (W_I)_{I \subseteq [k]}$, we let $\mathbf{W}_{\mathcal{P}} = ((W_I)_{\mathcal{P}})_{I \subseteq [k]}$.

In order to prove \Cref{thm:classical}, we first show that the set of all admissible graphon systems is a compact subset of the metric space $(\widetilde{\mathcal{W}}_0^k, \delta_{\square})$.
This is straightforward from the subsequent theorem because it is easy to check that the set of admissible graphon systems satisfies the two conditions of the theorem.

\begin{restatable}{theorem}{compact}\label{thm:compact_more}
    Let $X$ be a nonempty subset of $\widetilde{\mathcal{W}}_0^{(k)}$ satisfying:
    \begin{itemize}
        \item if $\mathbf{W} \in X$, then $\mathbf{W}_{\mathcal{P}} \in X$ for every finite partition $\mathcal{P}$ of $[0, 1]$;
        \item if a sequence of graphon systems $(\mathbf{W}^n)_{n\in \mathbb{N}}$ in $X$ component-wise converges to a graphon system $\mathbf{W}$ almost everywhere, then $\mathbf{W} \in X$.
    \end{itemize}
    Then $X$ is a compact subset of $(\widetilde{\mathcal{W}}_0^k, \delta_{\square})$.
\end{restatable}

We provide the proof of this theorem in Appendix~\ref{appendix:easy_proofs}.
The remaining part is to show that an admissible graphon system is the limit of a sequence of graph systems.
For an admissible graphon system $\mathbf{W}=(W_I)_{I \subseteq [k]}$ of order $k$ and an $n$-element tuple $S = (x_1, \ldots, x_n) \in  [0, 1]^n$, we define two tuples of weighted graphs on $n$ vertices as follows.
\begin{definition}~
 For each $I \subseteq [k]$, consider a weighted graph $W_I[S]$ on vertex set $[n]$ with the edge weight $W_I(x_i,x_j)$ on the edge $ij$. Let $\mathbf{W}[S]= (W_I[S])_{I\subseteq [k]}$.
 The tuple of weighted graphs obtained by deleting all the loops from each graph in $\mathbf{W}[S]$ is denoted by $\mathbf{H}_S(n, \mathbf{W})$.
\end{definition}
Note that $W_\emptyset[S]$ is always a complete graph with loops where all edge weights are $1$ as $W_{\varnothing} \equiv 1$.
Also, note that a weighted graph can be viewed as a system of step graphons whose values on each step are the corresponding edge weights.

Let $\mathbb{H}(n, \mathbf{W}) = \mathbf{H}_S(n,\mathbf{W})$ be a random variable obtained by choosing a tuple $S\in [0,1]^{n}$ uniformly at random. 
Then it is easy to see that the resulting random sample $\mathbf{H} =\mathbb{H}(n, \mathbf{W})$ is an admissible graphon system with probability $1$.


For an admissible system of weighted graphs $\mathbf{H}=(H_I)_{I \subseteq [k]}$, define a random graph system $\mathbb{G}(\mathbf{H})$ of order $k$ as follows. Considering each $H_I$ as a graphon, \Cref{def: barW} yields new graphons $\overline{H}_I$, which can be again considered to be a weighted graph.
For each $uv\in \binom{[n]}{2}$, consider an independent random variable $I_{uv}$ that becomes a set $I\subseteq [k]$ with probability $\overline{H}_I(uv)$, the edge weight of $uv$ in the weighted graph $\overline{H}_I$. 
This random variable is well-defined as $\mathbf{H}$ is admissible, i.e., $\overline{H}_I(uv) \in [0, 1]$ and $\sum_{I\subseteq [k]} \overline{H}_I(uv) = 1$ for each $uv\in \binom{[n]}{2}$.
We define a random graph system $\mathbb{G}(\mathbf{H})= (G_i)_{i\in [k]}$ as $uv\in E(G_i)$ if and only if $i\in I_{uv}$.
This yields a random graph system, where the probability that $uv\in \bigcap_{i\in I} E(G_i)$ is equal to $\sum_{J\supset I} \overline{H}_J(uv) = H_I(uv)$ for all $I\subseteq [k]$.
We denote by $\mathbb{G}(n,\mathbf{W})$ the random variable $\mathbb{G}(n,\mathbb{H}(n,\mathbf{W}))$.
Note that there are two levels of randomness, a choice of $S$ yielding $\mathbf{H}=\mathbb{H}(n,\mathbf{W})$ and choices of $I_{uv}$ yielding $\mathbb{G}(n,\mathbf{W})=\mathbb{G}(\mathbf{H})$.
When $k=1$, these definitions are identical to those of $\mathbb{H}$ and $\mathbb{G}$ for the graphons defined in~\cite{Lovasz2012}. 
Similarly as graphons, we can show the following theorem stating that our random graph system $\mathbb{G}(n, \mathbf{W})$ is very close to $\mathbf{W}$ in cut-distance with high probability.
We supply the proof of this theorem in Appendix~\ref{sec:W_random_graph}.

\begin{theorem}\label{thm:W_random_graph}
     With probability $1-o(1)$, $$\delta_{\square}(\mathbf{W}, \mathbb{G}(n, \mathbf{W})) \leq \frac{10^3 \times 8^k}{\sqrt{\log n}}.$$
\end{theorem}

\begin{proof}[Proof of \Cref{thm:classical}]
    By Theorem~\ref{thm:compact_more}, the set of admissible graphon systems is compact.
    For an admissible graphon system $\mathbf{W}$ of order $k$ and $\varepsilon>0$, there exists $n \in \mathbb{Z}_{>0}$ such that $\delta_{\square}(\mathbf{W}, \mathbb{G}(n, \mathbf{W}))<\varepsilon$ with high probability by \Cref{thm:W_random_graph}.
    Therefore, one can choose a sequence of graph systems of order $k$ that converges to $\mathbf{W}$.
\end{proof}

\section{Rainbow Tur\'{a}n density}\label{sec:applications}

We start this section by providing an elementary proof of \Cref{thm:limit_of_turan_density}
\begin{proof}[Proof of \Cref{thm:limit_of_turan_density}]
    Let $\pi =  1-\frac{1}{\chi(\mathrm{sim}(H))-1}$.
    Fix $\varepsilon>0$ and suppose that $k \geq \lceil 2\varepsilon^{-1} |E(H)| \rceil+1$.
    We claim that $\pi_k^{\ast}(H) \leq \pi+\varepsilon$.
    Since $\pi_k^{\ast}(H) \geq \pi$ for any $k$, it concludes the proof.
    Let $n$ be a sufficiently large integer and $(G_1, \ldots, G_k)$ be a graph system on $n$ vertices with $|E(G_i)| \geq (\pi + \varepsilon){n \choose 2}$ for every $i \in [k]$. Let $G'$ be a graph on the same vertex set such that $uv$ is an edge of $G'$ if and only if $uv \in E(G_i)$ for at least $|E(H)|$ indices $i \in [k]$.
    Then we have
    $$k \cdot (\pi + \varepsilon) {n \choose 2} \leq \sum_{i \in [k]} |E(G_i)| \leq k \cdot |E(G')| + (|E(H)|-1) \cdot {n \choose 2}.$$
    Thus we get $|E(G')| \geq (\pi + \frac{\varepsilon}{2}){n \choose 2}$. As $n$ is sufficiently large, the graph $G'$ contains $\mathrm{sim}(H)$.
    By choosing colors greedily, it ensures that $(G_1, \ldots, G_k)$ contains a rainbow copy of $H$, thereby satisfying the inequality $\pi_k^{\ast}(H) \leq \pi+\varepsilon$.
\end{proof}

We now use the theory on graphon systems to prove \Cref{thm:existence}, \Cref{thm:bipartite_max} and \Cref{thm:supersaturation}. 
For this purpose, we need a colored version of the graph removal lemma.
As noted in Fox~\cite{Fox2011}, it was suggested that his proof could be adapted to demonstrate the colored version.
For completeness, we provide a proof of this lemma in Appendix~\ref{sec:H-removal}, following the strategy outlined in~\cite{Lovasz2012}.

\begin{definition}
    Let $(H,\psi)$ be a rainbow pre-coloring tuple.
    For a graphon system $\mathcal{G}=(G_1, \ldots, G_k)$ on $V$, a \emph{rainbow copy} of $(H, \psi)$ in $\mathcal{G}$ is a subgraph $H'$ of complete graph on $V$ together with an injective map $\psi':E(H) \to [k]$ such that the following holds:
    there exists an isomorphism $\phi:H\rightarrow H'$ such that $\psi'\circ \phi: E(H)\rightarrow [k]$ is an extension of $\psi$ and $e \in E(G_{\psi'(e)})$ for every $e \in E(H)$.
\end{definition}

\begin{theorem}[Rainbow $H$-removal lemma]\label{thm:H_removal}
    Let $(H,\psi)$ be a rainbow pre-coloring tuple.
    For every $\varepsilon>0$, there exists $\delta>0$ such that if a graph system $\mathcal{G} = (G_1, \ldots, G_k)$ has at most $\delta |V(\mathcal{G})|^{|V(H)|}$ rainbow copies of $(H, \psi)$, then one can delete at most $\varepsilon n^2$ edges in total to make $\mathcal{G}$ rainbow $(H, \psi)$-free.
\end{theorem}
The same strategy can be employed to establish the next theorem.
\begin{theorem}\label{thm:homomorphic_H_removal}
    Let $(H,\psi)$ be a rainbow pre-coloring tuple.
    For every $\varepsilon>0$, there exists $\delta>0$ such that if a graph system $\mathcal{G} = (G_1, \ldots, G_k)$ has at most $\delta |V(\mathcal{G})|^{|V(H)|}$ rainbow copies of color homomorphic images of $(H, \psi)$, then we can delete at most $\varepsilon n^2$ edges in total to make $\mathcal{G} = (G_1, \ldots, G_k)$ have no rainbow copies of color homomorphic images of $(H, \psi)$.
\end{theorem}

We are now ready to derive \Cref{thm:existence}.
Because the set of admissible graphon systems form a compact subset of $\widetilde{\mathcal{W}}^{(k)}_0$, the following theorem immediately implies \Cref{thm:existence}.
\begin{theorem}\label{thm:stronger_existence}
    We have 
     $$\pi_k^{\ast}(H, \psi) = \sup_{t^{*}_{(H, \psi)}(\mathbf{W})=0} \min_{1 \leq \ell \leq k} t_{K_2}(W_\ell),$$ 
    where the supremum is taken over all admissible graphon systems $\mathbf{W}$ of order $k$ with $t^{*}_{(H, \psi)}(\mathbf{W})=0$.
\end{theorem}
\begin{proof}
    Let $\pi$ denote the right-hand side of the desired equality.
    As the functions $t^{*}_{(H, \psi)}(\mathbf{W})$ and $\min_{1 \leq \ell \leq k} t_{K_2}(W_\ell)$ are continuous, by compactness there exists a graphon system $\mathbf{W}$ of order $k$ that attains the supremum.

    Fix $\varepsilon>0$ and choose $\delta>0$ small enough. 
    By Theorem~\ref{thm:classical} and Theorem~\ref{thm:counting_lemma}, for sufficiently large $n$, there exists a graph system $\mathcal{G}$ such that $t^{*}_{(H, \psi)}(\mathcal{G}) \leq \delta$ and $\min_{1 \leq \ell \leq k} t_{K_2}(G_\ell) \geq \pi-\delta$.
    Then by Theorem~\ref{thm:H_removal}, one can delete at most $\varepsilon |V(\mathcal{G})|^2$ edges from $\mathcal{G}$ to make it rainbow $(H, \psi)$-free.
    Since $t_{K_2}(G_\ell) = \frac{2|E(G_\ell)|}{|V(\mathcal{G})|^2}$, we have $\mathrm{ex}_k^{\ast}(n, H, \psi) \cdot {n \choose 2}^{-1} \geq \pi-\delta-2\varepsilon$.
    
    To show the opposite inequality, again fix $\varepsilon>0$. Suppose that $\limsup_{n \rightarrow \infty} \mathrm{ex}_k^{\ast}(n, H, \psi) \cdot {n \choose 2}^{-1} \geq \pi + 2\varepsilon$.
    Then there exists an increasing sequence of integers $(n_i)$ such that for each $i$, there exists an $n_i$-vertex graph system $\mathcal{G}^i$ that is rainbow $(H, \psi)$-free with $\min_{1 \leq \ell \leq k} |E(G_{\ell}^i)| \cdot {n_i \choose 2}^{-1} \geq \pi+\varepsilon$.
    We may assume that $n_1$ is large enough that $\min_{1 \leq \ell \leq k} t_{K_2}(G_{\ell}^i) \geq \pi+\frac{\varepsilon}{2}$.
    
    We assert that there exists $\delta = \delta(\varepsilon)>0$  such that $t^{*}_{(H, \psi)}(\mathcal{G}^i) \geq \delta$ holds for all $i\in \mathbb{N}$.
    If not, then we can take a subsequence $(n_{s_i})$ of $(n_i)$ with $t^{*}_{(H, \psi)}(\mathcal{G}^{s_i}) \rightarrow 0$.
    By compactness, there exists a further subsequence that converges to an admissible graphon system $\mathbf{W}$.
    Then $\mathbf{W}$ satisfies that $t^{*}_{(H, \psi)}(\mathbf{W})=0$ but $\min_{1 \leq \ell \leq k} t_{K_2}(W_\ell) \geq \pi+\frac{\varepsilon}{2}$, which contradicts the definition of $\pi$.
    In conclusion, there exists $\delta>0$ such that $t^{*}_{(H, \psi)}(\mathcal{G}^i) \geq \delta$ for each $i\in \mathbb{N}$.
    
    On the other hand, as $\mathcal{G}^i$ is rainbow $(H, \psi)$-free, $\mathcal{G}^i$ can only contain degenerate copies of $(H, \psi)$.
    Therefore, $t^{*}_{(H, \psi)}(\mathcal{G}^i) = O(\frac{1}{n_i})$, which leads to a contradiction.
\end{proof}
The following corollary confirms \Cref{thm:supersaturation}. 
\begin{corollary}[Supersaturation]
    For every $\varepsilon>0$, there exists $\delta>0$ such that if an $n$-vertex graph system $\mathcal{G}$ has at least $(\pi_k^{\ast}(H, \psi) + \varepsilon){n \choose 2}$ edges, then it has at least $\delta n^{|V(H)|}$ copies of $(H, \psi)$.
\end{corollary}
\begin{proof}
    As $\mathcal{G}$ cannot be made rainbow $(H, \psi)$-free by deleting $\frac{\varepsilon n^2}{2}$ edges, the conclusion comes from Theorem~\ref{thm:H_removal}.
\end{proof}

\begin{remark}\label{rmk1}
    We note that \Cref{thm:stronger_existence} can be extended to an arbitrary continuous function $h:\widetilde{W}_0^{(k)} \rightarrow \mathbb{R}$.
More precisely, the limit of the minimum of $h(\mathcal{G})$ among all graph systems of order $k$ on $n$ vertices without $(H, \psi)$ is equal to the infimum of $h(\mathbf{W})$ among all graphon systems of order $k$ with $t_{(H, \psi)}^{\ast}(\mathbf{W})=0$.
The proof simply follows the proof of \Cref{thm:stronger_existence} together with the fact that by the uniform continuity of $h$, deleting $\varepsilon' |V(\mathcal{G})|^2$ edges changes the value of $h$ by at most $\varepsilon$ when $\varepsilon'\ll \varepsilon$. 
Furthermore, one can further show that the rainbow Tur\'an density of finite family of pre-coloring tuples likewise exists.
\end{remark}

\begin{theorem}\label{thm:density_hom_image}
    For every rainbow pre-coloring tuple $(H, \psi)$,
    $$\pi_k^{\ast}(H, \psi) = \pi_k^{\ast}(\{\text{all color homomorphic images of }(H, \psi)\}).$$
\end{theorem}
\begin{proof}
    Let $\pi$ denote the right-hand side of the desired equality.
    Considering $(H, \psi)$ is a color homomorphic image of itself,  $\pi_k^{\ast}(H, \psi) \geq \pi$ is clear.
    To show the opposite inequality, suppose that $\pi_k^{\ast}(H, \psi) \geq \pi+\varepsilon$ for fixed $\varepsilon > 0$.
    Then for a sufficiently large $n$, there exists a rainbow $(H, \psi)$-free graph system $\mathcal{G}^n$ on $n$ vertices such that $|E(G_{i}^n)| \geq (\pi + \frac{\varepsilon}{2}){n \choose 2}$ for each $i \in [k]$ and
    there are at most $o(n^{|V(H)|})$ distinct homomorphic images of $(H, \psi)$.
    By Theorem~\ref{thm:homomorphic_H_removal}, one can delete at most $\frac{\varepsilon n^2}{4}$ edges from $\mathcal{G}^n$ to make it have no color homomorphic images of $(H, \psi)$.
    Consequently the resulting graph system has at least $(\pi + \frac{\varepsilon}{4}) {n \choose 2}$ edges and does not contain color homomorphic images of $(H,\psi)$, which is a contradiction to the definition of $\pi$. 
\end{proof}

Consequently, this theorem yields the following corollary, which immediately implies \Cref{thm:bipartite_max}.

\begin{corollary}
    Let $H$ be a bipartite multigraph with $\ell$ edges equipped with a rainbow pre-coloring $\psi$. For $k \geq \ell$, we have $$\pi^{\ast}_k(H,\psi) \leq \pi^{\ast}_k(2 \text{ vertices with $\ell$ parallel edges}) = \frac{\ell-1}{k}.$$
\end{corollary}
\begin{proof}
    The first inequality comes from Theorem~\ref{thm:density_hom_image}.
    To show the second equality, we first show the lower bound. 
    Let $F$ be the graph with two vertices and $\ell$ parallel edges.
    Let $\mathbf{W}=(W_I)_{I \subseteq[k]}$ be the graphon system defined by 
    \begin{align*}
        W_I \equiv 
        \begin{cases}
            1, & \text{if } I = \emptyset, \\
            \frac{{k-|I| \choose \ell-1-|I|}}{{k \choose \ell-1}}, & \text{if } 0<|I| < \ell, \\
            0, & \text{otherwise}.
        \end{cases}
    \end{align*}
    In particular, if $|I|=1$, then $W_I \equiv  \frac{{k-1 \choose \ell-2}}{{k \choose \ell-1}} = \frac{\ell-1}{k}$.
    Since $W_I \equiv 0$ for every $I \subseteq [k]$ with $|I| \geq \ell$, we have $t^{*}_F(\mathbf{W})=0$.
    Finally, we can check that $\mathbf{W}$ is admissible by the direct computation: $\overline{W}_I \equiv \frac{1}{{k \choose \ell-1}}$ if $|I|=\ell-1$ and $\overline{W}_I \equiv 0$ otherwise.

    For the upper bound, let $\mathbf{W} = (W_I)_{I \subseteq [k]}$ be an admissible graphon system with $t^{*}_F(\mathbf{W})=0$.
    Then $W_I \equiv 0$ for all $I \subseteq [k]$ with $|I| = \ell$.
    As $W_I = \sum_{J \supseteq I} \overline{W}_J$ and $\overline{W}_J \geq 0$, we have $\overline{W}_I \equiv 0$ and thus $W_I \equiv 0$ for all $I \subseteq [k]$ with $|I| \geq \ell$.
    Because $\sum_{I \subseteq [k]} \overline{W}_I = W_{\emptyset} \equiv 1$, we obtain 
    \begin{align*}
        \sum_{i=1}^k t_{K_2}(W_i) & = \sum_{i=1}^k \sum_{I \subseteq [k], i \in I} t_{K_2}(\overline{W}_I)
         = \sum_{I \subseteq [k], |I| \leq \ell-1} |I|\cdot t_{K_2}(\overline{W}_I)
        \leq \sum_{I \subseteq [k]} (\ell-1) t_{K_2}(\overline{W}_I) \leq \ell-1.
    \end{align*}
    Therefore, there exists an index $i \in [k]$ such that $t_{K_2}(W_i) \leq \frac{\ell-1}{k}$.
\end{proof}
This implies that for any rational number $r \in (0,1) \cap \mathbb{Q}$, there exists a multi-graph $H$ and an integer $k$ such that $\pi_k^{*}(H) = r$.

\section{Extremal structures of graph systems having no rainbow tree}\label{sec:structure}
In this section, we show that an extremal graph system having no rainbow tree $T$ admits a certain rigid structure.
We first collect some observations and definitions.
\begin{observation}\label{lem:simple_graphs}
    For every graph $H$ equipped with a rainbow pre-coloring $\psi$ and graphon system $\mathbf{W}=(W_I)_{I \subseteq [k]}$, we have $t^*_{(H, \psi)}(\mathbf{W}) = t^*_{(H, \psi)}(\mathrm{span}(W_1, \ldots, W_k))$.
\end{observation}

\begin{definition}
    For a tree $T$ with a root $u \in V(T)$ and a rainbow coloring $\psi:E(T)\rightarrow [k]$, the \textit{rooted density} of $(T, u, \psi)$ on a tuple of graphons $\mathbf{W} = (W_1, \ldots, W_k)$ at $x \in [0, 1]$ is 
    $$rt^{*}_{(T, u, \psi)}(\mathbf{W}, x) = \int_{[0, 1]^{|V(T)|-1}} \prod_{vw \in E(H): u \neq v, w} W_{\psi(vw)}(x_v, x_w)  \times \prod_{uv \in E(H)} W_{\psi(uv)}(x, x_v) \prod_{v \in V(T), v \neq u} dx_v.$$
    When $\psi$ is a rainbow pre-coloring, define 
    $$rt^{*}_{(T, v, \psi)}(\mathbf{W}, x) =  \sum_{\tilde{\psi}} rt^{*}_{(T,v, \tilde{\psi})}(\mathbf{W}), $$ 
    where the sum is taken over all injective functions $\tilde{\psi}:E(T) \rightarrow [k]$ extending $\psi$.

    For a graphon $W$, we define the \textit{degree} of $W$ at $x$ by
    $$d_{W}(x) := \int_{[0, 1]} W(x, y) dy.$$
    The rooted density and the degree are defined almost everywhere by the Fubini theorem. 
\end{definition}

A function $f:\widetilde{W}^{(k)}_0 \to \mathbb{R}$ is said to be \emph{monotone} if $f(\mathbf{W}) \leq f(\mathbf{U})$ whenever $W_I \leq U_I$ almost everywhere for every $I \subseteq [k]$.
A function $f:\widetilde{W}^{(k)}_0 \to \mathbb{R}$ is said to be \emph{simple} if $f(\mathbf{W}) = f(\mathrm{span}(W_1, \ldots, W_k))$ for any graphon system $\mathbf{W}=(W_I)_{I \subseteq [k]}$ of order $k$. 

Recall that for a graphon system $\mathbf{W}$ and a partition $\mathcal{P}$ of $[0, 1]$, we define $\mathbf{W}_{\mathcal{P}} := ((W_I)_{\mathcal{P}})_{I \subseteq [k]}$.
A graphon system $\mathbf{W}$ with $\mathbf{W} \equiv \mathbf{W}_{\mathcal{P}}$ for some partition $\mathcal{P}$ of $[0,1]$ is called a \emph{step graphon} with steps in $\mathcal{P}$.
If $|\mathcal{P}|=m$, then we say $\mathbf{W}$ has at most $m$ steps.
\begin{theorem}\label{thm:structure}
    Let $T$ be a tree and $(T, \psi)$ be a rainbow pre-coloring tuple. Then there is a constant $m\leq 2^{2|V(T)|}$ satisfying the following:
    for every continuous monotone simple function $h:\widetilde{W}^{(k)}_0 \to \mathbb{R}$, there exists an admissible graphon system $\mathbf{W}$ such that
    \begin{enumerate}
        \item $t^{*}_{(T, \psi)}(\mathbf{W})=0$;
        \item $\mathbf{W}$ is a maximizer of $h$ among all admissible graphon systems of order $k$ with $(T, \psi)$-density zero;
        \item $\mathbf{W}$ is a step graphon system with at most $m$ steps and with value either $0$ or $1$ almost everywhere.
    \end{enumerate}
\end{theorem}
\begin{proof}
    For each edge $e = uv \in E(T)$, $T-e$ is a disjoint union of two trees. Denote them by $T^e_\kappa$ where $\kappa \in T^e_\kappa$ for $\kappa \in \{ u,v \}$.
    For any injective map $\hat{\psi}:E(T) \rightarrow [k]$, let  $\hat{\psi}^e_v$ be the restriction of $\hat{\psi}$ to $E(T^e_v)$.

    Let $\mathcal{T}$ be the family of tuples $(T^e_v, v, \hat{\psi}^e_v)$ for all $v\in V(T)$, $e \in E(T)$ with $v \in e$ and $\hat{\psi}$ is a rainbow extension of $\psi$ to $E(T)$.

    As the set of admissible graphon systems is compact and $t^{*}_{(T, \psi)}$ is continuous, one can choose an admissible graphon system $\mathbf{W}$ that maximizes $h$ among all admissible graphon systems $\mathbf{W}$ with $t^{*}_{(T, \psi)}(\mathbf{W})=0$.
    Among all such maximizers, choose one that maximizes $\sum_{t=1}^k t_{K_2}(W_t)$. 
    For each tuple $(T', v, \hat{\psi}) \in \mathcal{T}$, consider the partition $\mathrm{supp}(rt^{*}_{(T', v, \hat{\psi})}(\mathbf{W}, \ast))$ and $[0, 1] \setminus \mathrm{supp}(rt^{*}_{(T', v, \hat{\psi})}(\mathbf{W}, \ast))$ of $[0, 1]$.
    Let $\mathcal{P}=\{V_1, \ldots, V_m\}$ be the common refinement of all such bipartitions over all choices of $(T', v, \hat{\psi})\in \mathcal{T}$, then $m\leq 2^{2|V(T)|}$.
    We claim that $\mathbf{W}$ is a step $0$-$1$ function with steps in $\mathcal{P}$ having at most $2^{2|V(T)|}$ steps.

    Suppose not. Then there are indices $\ell \in [k]$ and $i,j \in [m]$ such that $W_\ell$ is not constantly $1$ and not constantly $0$ almost everywhere on $V_i \times V_j$.
    As $t^{*}_{(T, \psi)}(\mathbf{W})=0$, we have
    $$\sum_{(uv, \hat{\psi})}\int_{(V_i \times V_j) \cup (V_j \times V_i)} rt^{*}_{(T^{uv}_v, v, \hat{\psi}^{uv}_v)}(\mathbf{W}, x) \cdot W_{\ell}(x, y) \cdot rt^{*}_{(T^{uv}_u, u, \hat{\psi}^{uv}_u)}(\mathbf{W}, y) dxdy=0,$$
    where the sum is taken over all $uv\in E(T)$ and all rainbow extensions 
    $\hat{\psi}:E(T)\rightarrow [|E(T)|]$ of $\psi$ with $\hat{\psi}(uv)=\ell$.

    Because of the definition of $\mathcal{P}$, if $rt^{*}_{(T^{uv}_v, v, \hat{\psi}^{uv}_v)}(\mathbf{W}, x) \cdot rt^{*}_{(T^{uv}_u, u, \hat{\psi}^{uv}_u)}(\mathbf{W}, y) > 0$ for a pair $(x,y) \in V_i \times V_j$, then the same holds for almost every $(x,y) \in V_i \times V_j$.
    Hence, as $W_{\ell}$ is not $0$ almost everywhere on $V_i \times V_j$, we have $rt^{*}_{(T^{uv}_v, v, \hat{\psi}^{uv}_v)}(\mathbf{W}, x) \cdot rt^{*}_{(T^{uv}_u, u, \hat{\psi}^{uv}_u)}(\mathbf{W}, y) = 0$ for almost every $(x,y) \in (V_i \times V_j) \cup (V_j \cup V_i)$.
    Define a new graphon system $\mathbf{W}' = \mathrm{span}(W_1, \ldots, W_\ell', \ldots, W_k)$ where $W_\ell'$ is equal to $W_\ell$ except that it has value $1$ on almost every point in $(V_i \times V_j) \cup (V_j \times V_i)$.
    Then for each rainbow extension $\hat{\psi}$ of $\psi$, if $\ell \notin \hat{\psi}(E(T))$, then $t^{*}_{(T, \hat{\psi})}(\mathbf{W}')=t^{*}_{(T, \hat{\psi})}(\mathbf{W})$;
    if $\ell \in \hat{\psi}(E(T))$, let $uv \in E(T)$ be the edge with $\hat{\psi}(uv)=\ell$.
    As two functions $\hat{\psi}^{uv}_v$ and $\hat{\psi}^{uv}_u$ does not use the color $\ell$, we have 
    $$rt^{*}_{(T^e_v, v, \hat{\psi}^e_v)}(\mathbf{W}, x) = rt^{*}_{(T^e_v, v, \hat{\psi}^e_v)}(\mathbf{W}', x) \text{ and } rt^{*}_{(T^e_u, u, \hat{\psi}^e_u)}(\mathbf{W}, y) = rt^{*}_{(T^e_u, u, \hat{\psi}^e_u)}(\mathbf{W}', y).$$
    Therefore, we have 
    \begin{align*}
        t^{*}_{(T, \hat{\psi})}(\mathbf{W}') & = \int_{[0, 1]^2} rt^{*}_{(T^{uv}_v, v, \hat{\psi}^{uv}_v)}(\mathbf{W}', x) \cdot W_{\ell}'(x, y) \cdot rt^{*}_{(T^{uv}_u, u, \hat{\psi}^{uv}_u)}(\mathbf{W}', y) dxdy \\
        & = \int_{[0, 1]^2 \setminus ((V_i \times V_j) \cup (V_j \times V_i))} rt^{*}_{(T^{uv}_v, v, \hat{\psi}^{uv}_v)}(\mathbf{W}, x) \cdot W_{\ell}(x, y) \cdot rt^{*}_{(T^{uv}_u, u, \hat{\psi}^{uv}_u)}(\mathbf{W}, y) dxdy  \\ &\qquad + \int_{(V_i \times V_j) \cup (V_j \times V_i)} rt^{*}_{(T^{uv}_v, v, \hat{\psi}^{uv}_v)}(\mathbf{W}, x) \cdot rt^{*}_{(T^{uv}_u, u, \hat{\psi}^{uv}_u)}(\mathbf{W}, y) dxdy  = 0.
    \end{align*}
    Therefore, $t^{*}_{(T, \psi)}(\mathbf{W}')=0$.
    As $h$ is monotone, $h(\mathbf{W})\leq h(\mathbf{W}')$ and $\sum_{t=1}^k t_{K_2}(W_t)<\sum_{t=1}^k t_{K_2}(W'_t)$, a contradiction.
    Therefore, $\mathbf{W}$ is a step graphon system, and it has at most $2^{|V(T)|}$ steps by definition.
\end{proof}

By using some facts from real algebraic geometry, we deduce \Cref{thm:algebraic_number}.
For the notions in the following proof, see Section~\ref{appendix:real_algebraic_geometry}.

\begin{corollary}
    For a rainbow pre-coloring tuple $(T, \psi)$ with $T$ a tree, $\pi_k^{\ast}(T, \psi)$ can be computed by solving finitely many polynomial optimization problems.
    In particular, $\pi_k^{\ast}(T, \psi)$ is an algebraic number.
\end{corollary}

\begin{proof}
    Consider graph systems $\mathcal{G}=(G_1, \ldots, G_k)$ with $|V(\mathcal{G})| = m \leq 2^{|V(T)|}$, possibly with loops.
    Let $\mathbb{G}_T$ be the collection of all such graph systems with $t_{(T, \psi)}(\mathcal{G}) =0$.
    For each graph system $\mathcal{G} \in \mathbb{G}_T$, we associate classical graphon systems $\mathbf{W}$ with $m$ steps $I_1, \ldots, I_{m}$ of sizes $x_1, \dots, x_m \in [0,1]$ with $\sum_i x_i = 1$ such that $W_i \equiv 1$ on $I_{\ell} \times I_{\ell'}$ if and only if $\ell$ and $\ell'$ are adjacent in $G_i$, and $W_i \equiv 0$ on $I_{\ell} \times I_{\ell'}$ otherwise.
    By Theorem~\ref{thm:structure}, one of these graphon systems is a maximizer.

    For a fixed $\mathcal{G} \in \mathbb{G}_T$, the maximum possible value of $\min_{i \in [k]} t_{K_2}(W_i)$ among all associated graphon systems can be computed by solving a polynomial optimization problem
    $$
     \sup_{\substack{x_1 + \cdots + x_m = 1 \\ x_i \geq 0}} \min \{ f_1(x_1,\dots, x_m), \ldots, f_k(x_1,\dots,x_m) \},
    $$
    for certain (homogeneous) polynomials $f_1, \ldots, f_k$ of degree 2 with coefficients 0 or 1, where each $x_i$ represents the length of the $i$-th step.
    Here, each $f_i$ corresponds to the edge density of $W_i$.
    By taking the maximum of those numbers over all $\mathcal{G} \in \mathbb{G}_T$, we can compute $\pi_k^{\ast}(T, \psi)$.
    It suffices to show that the optimization problem results in an algebraic number.

    Let $\Sigma = \{ (x_i) \in \mathbb{R}^m \mid x_1+\cdots+x_m = 1,\> x_i \geq 0 \text{ for } i \in [m] \}$ be the constraint set. Let $f(x) = \min \{ f_1(x), \ldots, f_r(x) \}$. As $\Sigma$ is compact, the supremum can be attained at a point $x_0 \in \Sigma$. By the transfer principle (\Cref{thm:transfer_principle}), the supremum taken over the semialgebraic set $\mathbb{R}_{\mathrm{alg}}^m \cap \Sigma$ can be achieved at another point $x_1 \in \mathbb{R}_{\mathrm{alg}}^m \cap \Sigma$. It is clear that $f(x_0) \geq f(x_1)$. For any $n \in \mathbb{N}$, choose $p_n \in \mathbb{Q}^m$ with $|x_0-p_n| < \frac{1}{n}$. Again by the transfer principle, one can choose $y_n \in \mathbb{R}_{\mathrm{alg}}^m \cap \Sigma$ with $|y_n-p_n| < \frac{1}{n}$, so that $y_n$ converges to $x_0$.
    As $f(y_n)$ converges to $f(x_0)$, we have $f(x_0) \leq f(x_1)$.
    Therefore, $f(x_0) = f(x_1)$ and it is an algebraic number.
\end{proof}

\section{Proof of \Cref{thm:tree_less_star_2}}\label{sec:proof-tree-star}
We verify \Cref{thm:tree_less_star_2} in this section.

\begin{definition}\label{def:gamma}
    For a graphon $W$, define $\gamma(W) := 1 - \sqrt{t_{K_2}(W)}$.
\end{definition}
In fact, $\gamma(W)$ measures the possible maximum measure of a subset $A\subseteq [0,1]$ such that
there exists a graphon $W'$ with $t_{K_2}(W') = t_{K_2}(W)$ and $W'(x,y)=0$ for all $x\in A\times [0,1]$.
In terms of graphs, this is equivalent to `the maximum number of isolated vertices' of a graph with $|V(G)|$ vertices and $|E(G)|$ edges.
The following observation captures the above intepretation of $\gamma(W)$.

\begin{definition}
    Let $W$ be a graphon. Let $I(W)$ be the set of points $x\in [0, 1]$ at which $W$ has zero degree.
\end{definition}

\begin{observation}\label{obs:fraction-isolated}
    $\gamma(W) \geq \mu(I(W))$.
\end{observation}

\begin{proof}
    We have
     \[
      t_{K_2}(W) \leq \int_{([0, 1]\setminus I(W))^2} W(x, y)dxdy \leq (1 - \mu(I(W)))^2,
     \]
    and consequently $\mu(I(W)) \leq 1 - \sqrt{t_{K_2}(W)} = \gamma(W).$
\end{proof}

First, we determine the Tur\'an density of the star $K_{1,k}$.
\begin{lemma}\label{lem:star-isolated}
    Let $\mathbf{W} = \mathrm{span}(W_1, \dots, W_k)$ be a graphon system. If $\sum_{i\in [k]} \gamma(W_{i}) < 1$, then $t^{*}_{K_{1,k}}(\mathbf{W}) > 0$. As a consequence, for every positive integer $k$, the rainbow turan density of the $k$-edge star is $\pi^{\ast}_k(K_{1,k}) = \left(1 - \frac{1}{k}\right)^2.$
\end{lemma}

\begin{proof}
    Let $A = [0, 1]\setminus \bigcup_{i\in [k]} I(W_{i}).$ By \Cref{obs:fraction-isolated}, the inequality $\sum_{i\in [k]} \mu(I(W_{i})) \leq \sum_{i\in [k]} \gamma (W_{i}) < 1$ holds, which implies $\mu(A) > 0.$ Label the leaves of $K_{1,k}$ by $1, \ldots, k$ and the center by $k+1$. Let $\psi$ be a rainbow pre-coloring of $K_{1,k}$ that colors the edge $i(k+1)$ by $i$ for each $i\in [k].$
    Then 
    \begin{align*}
        t^{*}_{K_{1,k}}(\mathbf{W}) &\geq t^{*}_{(K_{1,k}, \psi)}(\mathbf{W})
        = \int_{[0,1]^{k+1}} \prod_{i\in [k]}W_{i}(x_i, x_{k+1}) \\
        &\geq \int_{[0, 1]} \prod_{i\in [k]}d_{W_{i}}(x_{k+1}) 
        \geq \int_A \prod_{i\in [k]}d_{W_{i}}(x_{k+1})
        > 0.
    \end{align*}
    The last inequality holds because $\mu(A) > 0$ and $d_{W_{i}} > 0$ on $A$ for each $i \in [k]$. This confirms the first statement.

    Assume that $\pi^{\ast}(K_{1,k}) > \left(1 - \frac{1}{k} \right)^2.$ By \Cref{thm:stronger_existence} and \Cref{lem:simple_graphs}, there is a graphon system $\mathbf{W}' = \mathrm{span}(W'_1, \dots, W'_k)$ such that $t^{*}_{K_{1,k}}(\mathbf{W}') = 0$ and $t_{K_2}(W'_{i}) > \left(1 - \frac{1}{k} \right)^2$ for each $i \in [k]$. Considering $\sum_{i\in [k]}\gamma (W'_{i}) < 1$, we have $t^{*}_{K_{1,k}}(\mathbf{W}') > 0$ by the first statement. This is a contradiction, and thus $\pi^{\ast}(K_{1,k}) \leq \left(1 - \frac{1}{k} \right)^2.$

    To complete the proof, it is enough to find a graphon system $\mathbf{W}' = \mathrm{span}(W'_1, \dots, W'_k)$ such that $t^{*}_{K_{1,k}}(\mathbf{W}') = 0$ and $t_{K_2}(W'_i) \geq \left(1 - \frac{1}{k} \right)^2$ for each $i\in [k]$. 
    Set $A_i =(0, 1) \setminus [\frac{i-1}{k}, \frac{i}{k}]$. Define a step graphon $W'_i$ by letting $W'_i(x, y) = 1$ if $(x, y)\in A_i^2$ and $0$ otherwise.
    Then we have $t_{K_2}(W'_i) = \mu(A_i^2) = \left(1 - \frac{1}{k} \right)^2.$ Let $\mathbf{W}' = \mathrm{span}(W'_1, \dots, W'_k).$ For any point $x\in [0, 1]$, there is an index $i\in [k]$ with $d_{W'_i}(x) = 0$, so $t^{*}_{K_{1,k}}(\mathbf{W}') = 0.$ Therefore, $\pi^{\ast}(K_{1,k}) = \left(1 - \frac{1}{k} \right)^2$. 
\end{proof}

Now we show that $\pi^*_k(T)\leq (\frac{k-1}{k})^2$ holds for all $k$-edge star $T$.
We intend to analyze how the rainbow Tur\'an density changes when we remove a star from a tree, ensuring that the remaining graph is also a tree.
For this, the subsequent concept will be beneficial.

\begin{definition}
    Let $T$ be a tree which is not a star. A \textit{leaf-star} $S$ of $T$ is a subgraph of $T$ that satisfies the following:
    \begin{enumerate}
        \item[$(1)$] $S$ is a star,
        \item[$(2)$] if $v$ is a center of $S$, then there is a vertex $u\in N_T(v) \setminus S$ such that every $u'\in N_T(v)\setminus \{u\}$ is a leaf of both $S$ and $T.$
    \end{enumerate}
    A leaf-star $S$ of $T$ has \textit{size} $\ell$ if $|N_T(v)\setminus \{u\}| = \ell.$
    If we want to emphasize $u$ and $v$, we call $S$ an \textit{$(u, v)$-leaf-star}.
\end{definition}
Later in the proof of \Cref{thm:tree_less_star_2}, we will handle two cases separately where $S$ has exactly $k-2$ edges or less than $k-2$ edges, where $|E(T)|=k$. The next two lemma will be encapsulate the proof of the former case where $T$ is a star with one edge subdivided.

\begin{lemma}\label{lem:P_3-first-leaf-star}
    Let $\mathbf{W} = \mathrm{span}(W_1, W_2, W_3)$ be a graphon system with $\gamma(W_1) + \gamma(W_2) + \gamma(W_3) < 1.$
    Let $P_3=wxyz$ be a $3$-edge path, and let $\psi:\{wx\}\rightarrow [1]$ be its rainbow pre-coloring.
    Then $t^{*}_{(P_3, \psi)}(\mathbf{W}) > 0.$    
\end{lemma}

\begin{proof}
    Suppose that we are given a graphon system $\mathbf{W} = \mathrm{span}(W_1, W_2, W_3)$ with $t^{*}_{(T, \psi)}(\mathbf{W}) = 0$.
    For each $e = (e_1, e_2, e_3) \in \{0, 1\}^3$, let $A_e$ be the set of points $v \in [0,1]$ such that $d_{W_i}(v) > 0$ precisely when $e_i = 1$ for each $i\in [3]$. Then $\{A_e: e\in \{0,1\}^3 \}$ forms a partition of the interval $[0,1]$ up to a measure zero set.
    Then we have
    $$W_i \leq \sum_{e: e_i = 1} \mathbbm{1}_{A_{e}\times A_{e}} + \sum_{\substack{e,e':e \neq e' \\ e_i = e'_i = 1}} \mathbbm{1}_{(A_{e}\times A_{e'}) \cup (A_{e'}\times A_{e})},$$
    almost everywhere.
    For simplicity, write $A_{e_1 + 2e_2 + 4e_3}$ to denote $A_{(e_1, e_2, e_3)}$. For instance, $A_5= A_{(1,0,1)}$.
This definition yields a structure on each graphon $W_i$ by forcing it to be zero on certain region $A_j\times A_\ell$. The structure is depicted in Figure~1 by omitting certain colors in some regions indicating that $W_i$ is zero there. In particular, this structure yields the following claim.
    
    \begin{claim}\label{certain-empty}
        $W_2 \equiv 0$ and $W_3 \equiv 0$ on $(A_7\times A_7) \cup (A_6\times A_7) \cup (A_7\times A_6)$. Also, $W_2 \equiv 0$ on $(A_3\times (A_6 \cup A_7)) \cup ((A_6 \cup A_7)\times A_3)$ and $W_3 \equiv 0$ on $(A_5\times (A_6 \cup A_7)) \cup ((A_6 \cup A_7)\times A_5).$
    \end{claim}

    \begin{claimproof}
        Let $X = (A_3 \cup A_7) \times (A_6 \cup A_7)$. Then
        $$0 = t^{*}_{(P_3, \psi)}(\mathbf{W}) \geq \int_{X} d_{W_1}(x)W_2(x, y)d_{W_3}(y) dx dy.$$
        Since $d_{W_1}(x)$ and $d_{W_3}(y)$ are strictly positive for $(x,y) \in X$, we have $W_2 \equiv 0$ on $X$. As $W_2$ is symmetric, the statement for $W_2$ follows. A similar argument also applies to $W_3$.
    \end{claimproof}
    
    \begin{figure}[ht]
        \includegraphics[height=5.5cm]{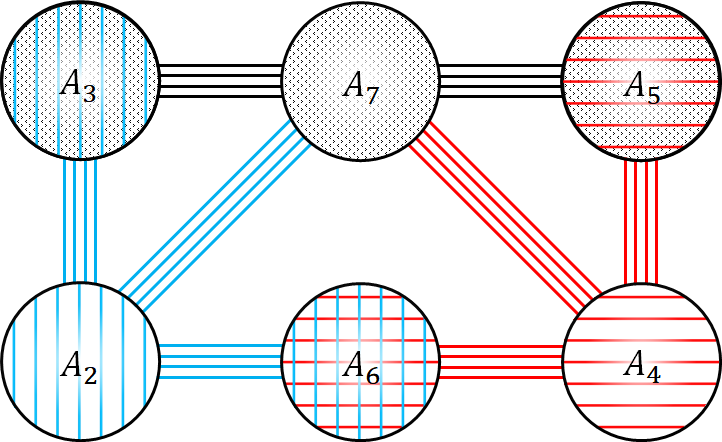}
        \centering
        \caption{The scheme of an extremal graphon system that is rainbow $(P_3,\psi)$-free. The ``edges'' with color 1, 2, and 3 are colored black, cyan, and red, respectively.}
    \end{figure}

    Let $m_i := \mu(A_i)$ for each $i \in \{0, 1, \dots, 7\}.$ Because
    \begin{align*}
        W_1 &\leq \mathbbm{1}_{(A_1 \cup A_3 \cup A_5 \cup A_7)\times (A_1 \cup A_3 \cup A_5 \cup A_7)}, \\
        W_2 &\leq \mathbbm{1}_{(A_2\cup A_3)\times (A_2\cup A_3)} + \mathbbm{1}_{(A_2\times A_6)\cup (A_6\times A_2)} + \mathbbm{1}_{(A_2\times A_7)\cup (A_7\times A_2)} + \mathbbm{1}_{A_6\times A_6},\\
        W_3 &\leq \mathbbm{1}_{(A_4\cup A_5)\times (A_5\cup A_4)} + \mathbbm{1}_{(A_4\times A_6)\cup (A_6\times A_4)} + \mathbbm{1}_{(A_4\times A_7)\cup (A_7\times A_4)} + \mathbbm{1}_{A_6\times A_6},
    \end{align*}
    we obtain the inequalities below.
    \begin{align}
        \gamma(W_1) &\geq f_1 := 1 - m_1 - m_3 - m_5 - m_7, \label{eq7}\\
        \gamma(W_2) &\geq f_2 := 1 - \sqrt{(m_2 + m_3)^2 + 2m_2(m_6 + m_7) + m_6^2}, \label{eq8}\\
        \gamma(W_3) &\geq f_3 := 1 - \sqrt{(m_4 + m_5)^2 + 2m_4(m_6 + m_7) + m_6^2}. \label{eq9}
    \end{align}

    Suppose on the contrary that there is a graphon system $\mathbf{W} = \mathrm{span}(W_1, W_2, W_3)$ such that $\gamma(W_1) + \gamma(W_2) + \gamma(W_3) < 1$ and $t^{*}_{(P_3,\psi)}(\mathbf{W}) = 0$. Our goal is to show that $F := f_1 + f_2 + f_3 \geq 1.$ Viewing $F$ as a (continuous) function of $m_1,\dots,m_7$, its domain is
    $$\left\{ (m_1, \dots, m_7) \in \mathbb{R}^7 \,\left\mid\, \sum_{i\in [7]} m_i \leq 1 \text{ and } m_i \geq 0 \right. \right\},$$
    which is a compact set. Henceforth, the minimum is attained, say at a point $(a_1, \dots, a_7)$. It is clear that $\sum_{i\in [7]} a_i = 1$ and $a_1 = 0.$ Let $\alpha$, $\beta$, and $\gamma$ be the values of $f_1$, $f_2$, and $f_3$ at the point $(a_1, \dots, a_7)$, respectively.
    By the assumption, $\alpha + \beta + \gamma < 1.$

    \begin{claim}\label{clm:a6=0}
        We may assume that $a_6 = 0.$
    \end{claim}

    \begin{claimproof}
        Assume $a_6 > 0.$ Putting $a_3' = a_3 + a_6$ and $a_6' = 0$, the difference between the values of $F$ is
        \begin{align*}
        F(a_1, a_2, a_3', a_4, a_5, a_6', a_7) - F(a_1, \dots, a_7) 
        & \leq -a_6 + \sqrt{(a_4 + a_5)^2 + 2a_4(a_6 + a_7) + a_6^2} - \sqrt{(a_4 + a_5)^2 + 2a_4a_7} \\
        & = -a_6\left(1 - \frac{a_6 + 2a_4}{\sqrt{(a_4 + a_5)^2 + 2a_4(a_6 + a_7) + a_6^2} + \sqrt{(a_4 + a_5)^2 + 2a_4a_7}} \right) \\
        & \leq -a_6\left(1 - \frac{a_6 + 2a_4}{(a_6+a_4) + a_4} \right) = 0,
        \end{align*}
        provided that $a_6+2a_4 \neq 0$. If $a_6+2a_4 = 0$, the difference is bounded by $-a_6 \leq 0$.
        By the minimality, we may assume that $a_6 = 0.$
    \end{claimproof}

    \begin{claim}\label{clm:a2>0}
        We may further assume that $a_2 , a_4 > 0.$
    \end{claim}

    \begin{claimproof}
        Assume $a_2 = 0.$ 
        Then we have $\gamma(W_1) \geq 1-a_3-a_5-a_7$, $\gamma(W_2) \geq 1-a_3 \geq a_4 + a_5 + a_7$ and $\gamma(W_3) \geq 1-a_4-a_5 \geq a_3.$ Thus, $\gamma(W_1) + \gamma(W_2) + \gamma(W_3) \geq 1+a_4 \geq 1$, a contradiction. 
        By symmetry, we would get a contradiction if $a_4 = 0$. 
        It follows that both $a_2$ and $a_4$ should be strictly positive.     
    \end{claimproof}

    \begin{claim}\label{clm:a_7<1/3}
        We may assume that $a_7 < \frac{1}{3}.$
    \end{claim}

    \begin{claimproof}
        Assume $a_7 \geq \frac{1}{3}.$ Then
        \begin{align*}
            2\left(1 - \frac{\beta + \gamma}{2} \right)^2
            &\leq (1 - \beta)^2 + (1 - \gamma)^2
            = (a_2+a_3)^2+2a_2a_7+(a_4+a_5)^2+2a_4a_7 \\
            &\leq (a_2+a_3+a_4+a_5)^2 + 2a_7\alpha
            = (1-a_7)^2+2a_7\alpha.
        \end{align*}
        Because $\alpha + \beta + \gamma < 1$, we have $\frac{(1 + \alpha)^2}{2} < 2\left(1 - \frac{\beta + \gamma}{2} \right)^2.$ 
        On the other hand, as $a_7 \geq \frac{1}{3}$, we infer that $(1 - a_7)^2 + 2a_7\alpha \leq \max\left\{ 2\alpha, \frac{4 + 6\alpha}{9} \right\}$.
        Therefore $\frac{(1+\alpha)^2}{2} < \max\{2\alpha, \frac{4 + 6\alpha}{9}\}$, a contradiction as $0 \leq \alpha \leq 1.$
    \end{claimproof}

    By the above claims, either $a_2 + a_3 > \frac{1}{3}$ or $a_4 + a_5 > \frac{1}{3}$ holds. Without loss of the generality, we assume $a_2 + a_3 > \frac{1}{3} > a_7.$
    Putting $a_3' = a_3 + a_2$ and $a_2' = 0$, the difference between the values of $F$ is
    \begin{align*}
        F(a_1,a_2',a_3',a_4,\dots,a_7) - F(a_1, \dots, a_7)
        & = -a_2 + \sqrt{(a_2 + a_3)^2 + 2a_2a_7} - (a_2 + a_3) \\
        & = -a_2 \left(1 - \frac{2a_7}{\sqrt{(a_2 + a_3)^2 + 2a_2a_7} + (a_2+a_3)} \right) \\
        & \leq -a_2\left(1 - \frac{2a_7}{2(a_2+a_3)} \right) < 0.
    \end{align*}
    The last inequality follows because $a_2 > 0$ and $a_2 + a_3 > a_7.$ This contradicts the minimality and verifies the lemma.
\end{proof}

Using the above lemma for the $3$-edge path, we can prove the following lemma dealing with the one-edge-subdivided star.

\begin{lemma}\label{lem:P3-second-leaf-star}
    Let $\mathbf{W} = \mathrm{span}(W_1, \dots, W_k)$ be a graphon system. Let $T$ be a $k$-edge tree that has an $(u, v)$-leaf-star $S$ of size $k-2$. Let $C$ be a subset of $[k]$ of size $k-2$ and let $\psi:E(S) \rightarrow C$ be a rainbow pre-coloring of $T$.
    If $\sum_{i\in [k]} \gamma(W_i) < 1$, then $t^{*}_{(T, \psi)}(\mathbf{W}) > 0.$
\end{lemma}

\begin{proof}
    Note that the induced graph $T[V(T)\setminus V(S)]$ is a single edge. Let $V(S) = \{v, x_1, x_2, \dots ,x_{k-2}\}.$ Suppose that there exists a graphon system such that the statement is not true.
    Among such graphon systems, we choose $\mathbf{W} = \mathrm{span}(W_1, \dots, W_k)$ at which $t^{*}_{K_2}(\mathbf{W}) = \sum_{i\in [k]} t_{K_2}(W_i)$ is maximized. Thus $\sum_{i \in [k]} \gamma(W_i) < 1$ but $t^{*}_{(T, \psi)}(\mathbf{W}) = 0.$
    
    We may assume $C = [k-2]$ and $\psi(vx_i) = i.$
    Let $\alpha_i := \gamma(W_i)$ for $i\in [k-2]$ and $\beta := \gamma(W_{k-1})$, $\gamma := \gamma(W_k)$.
    Let $A = \bigcap_{i \in [k-2]} \mathrm{supp}(d_{W_i})$ and let $W'_1 = \mathbbm{1}_{A \times A}$.
    Let $\mathbf{W'} = \mathrm{span}(W_1', W_2' := W_{k-1}, W_3':=W_k)$ be a graphon system and $P_3$ be a path $xyzw$ of length $3$.
    As $\mu(A) \leq 1 - \sum_{i \in [k-2]}\gamma(W_i)$, we have 
     \[
      \gamma(W_1') + \gamma(W_2') + \gamma(W_3') \leq \sum_{i \in [k]}\gamma(W_i) < 1.
     \]
    Thus $t^{*}_{(P_3, \psi')}(\mathbf{W}')>0$ by \Cref{lem:P_3-first-leaf-star}, where $\psi'$ is the pre-coloring of $P_3$ with $\mathbf{dom}(\psi') = xy$ and $\psi'(xy) = 1$.
    Since
    $$t^{*}_{(P_3, \psi')}(\mathbf{W}) = \int_{[0,1]^3} d_{W_1'}(x) (W_{k-1}(x, y)W_k(y, z)+W_{k}(x, y)W_{k-1}(y, z)) dxdydz>0,$$
    we infer that
    $$\int_{[0,1]^2} W_{k-1}(x, y)W_k(y, z)+W_{k}(x, y)W_{k-1}(y, z)dydz>0,$$
    on a subset of $A$ that has a positive measure.
    Therefore, 
    $$t^{*}_{(T, \psi)}(\mathbf{W}) = \int_{[0, 1]^3} \prod_{i \in [k-2]} d_{W_i}(x) (W_{k-1}(x, y)W_k(y, z)+W_{k}(x, y)W_{k-1}(y, z)) dxdydz>0,$$
    from the definition of $A$, which is a contradiction.
\end{proof}

As the previous lemmas handle the star and the one-edge-subdivided star, we now have to consider the other trees. For those threes, the 
following two observations will be helpful for us.
\begin{observation}\label{obs:twoleafstar}
    Let $T$ be a tree that is not a star. By taking the longest path in $T$, one can easily see that $T$ has at least two leaf-stars.
\end{observation}

\begin{observation}\label{obs:structure-leafstar}
    Let $T$ be a tree with two distinct leaf-star $S$ and $S'$ where $T - S'$ has more than two edges.
    Then there are at least three edges which are not incident to leaves of $T$ or is contained in $S$.
\end{observation}

We now prove that for a given graphon system $\mathbf{W}$ with $t^{*}_{(T, \psi)}(\mathbf{W}) = 0$ admits a structure depicted in Figure~2. In particular, this yields the next lemma.

\begin{lemma}\label{lem:induction-structure}
    Let $T$ be a $k$-edge tree that is not a star and $S$ be a leaf-star of $T$ that is centered at $v$ and has leaves $v_1, \dots, v_m$. Let $u$ be the non-leaf neighbor of $v$.
    Let $\psi$ be a rainbow pre-coloring of $T$ such that $\psi(vv_i)=i$ for each $i\in [m]$ and $\psi(uv)=k$.
    Let $T'=T-S$ and $\psi'$ be the restriction of $\psi$ to $E(T')$.
    Let $\mathbf{W}$ be a graphon system of order $k$ with $t^{*}_{(T, \psi)}(\mathbf{W}) = 0.$ Let $X = \bigcap_{i\in [m]} \mathrm{supp}(d_{W_i}) \subseteq [0, 1]$ and $Y = \mathrm{supp}(rt^{*}_{(T', u, \psi')}(\mathbf{W}, \ast))$.

    Let
     \[
      A = X \cap Y,\qquad B = X \setminus Y,\qquad
      C = Y \setminus X, \qquad
      D = [0,1] \setminus (A \cup B \cup C),
     \]
    and
    $$W'_k = \mathbbm{1}_{(B \cup D)^2 \cup (C \cup D)^2 \cup (A \times D) \cup (D \times A)}.$$
    Then we have the following:
    \begin{itemize}
        \item $W_k \leq W'_k$ almost everywhere,
        \item $t^{*}_{(T, \psi)}(\mathrm{span}(W_1, \dots, W_{k-1}, W'_k)) = 0.$
    \end{itemize}
\end{lemma}

\begin{figure}[ht]
    \includegraphics[height=5.5cm]{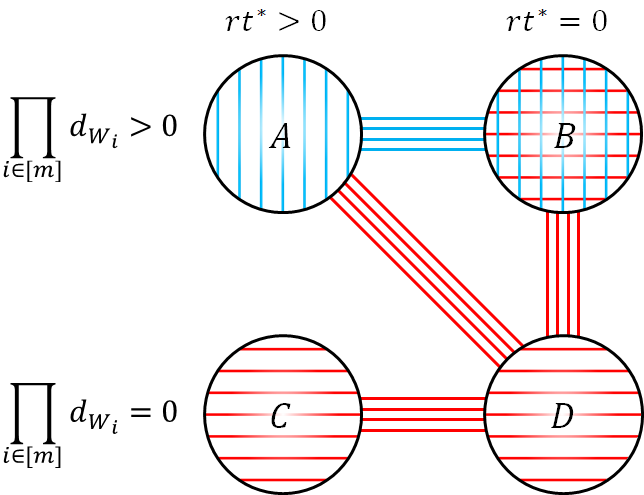}
    \centering
    \caption{The scheme of $(W_1, \dots, W_{k-1}, W_k')$.
    The ``edges'' corresponding to $W_k'$ are colored red, and the edges corresponding to $W_1, \dots, W_m$ are colored cyan.}
\end{figure}

\begin{proof}\label{lem:structure-graphon}
    Because $\psi$ is a rainbow pre-coloring, the image of $\psi'$ does not contain $[m]\cup \{k\}$.
    We have     
    \begin{align*}
        0=t^{*}_{(T, \psi)}(\mathbf{W}) &= \int_{[0,1]^{m+2}} \left(\prod_{i\in [m]}W_i(x_{v_i}, x_v)\right) W_k(x_{v}, x_u) rt^{*}_{(T', u, \psi')}(x_u) dx_vdx_{u} \prod_{i \in [m]} dx_{v_i}  \\
        & = \int_{[0,1]^2} \left(\prod_{i\in [m]} d_{W_i}(x_v)\right)W_k(x_{v}, x_u) rt^{*}_{(T', u, \psi')}(x_u) dx_vdx_{u},
    \end{align*}
    so the integrand is constantly zero almost everywhere.
    Observe that
    $$\left(\prod_{i\in [m]}d_{W_i}(x_v)\right)rt^{*}_{(T', v, \psi')}(x_u)>0,$$
    for $(x_v, x_u) \in (A \cup B) \times (A \cup C)$. As $W_k$ is symmetric, $W_k \equiv 0$ on $((A \cup B) \times (A \cup C)) \cup ((A \cup C) \times (A \cup B))$.
    Hence $W_k \leq  W'_k$ almost everywhere.
    On the other hand, if $W_k'(x, y) \neq 0$ for some $x,y \in [0,1]$, then either $d_{W_1}(x)=0$ or $rt^{*}_{(T', v, \psi')}(y)=0$ holds as $W_k' \equiv 0$ on $A \times A$.
    Therefore, 
     \[
      d_{W_1}(x_v) W_k'(x_v, x_u)rt^{*}_{(T', v, \psi')}(x_u) \equiv 0.
     \]
    This implies that $t^{*}_{(T, \psi)}(\mathrm{span}(W_1, \dots, W_{k-1}, W'_k)) = 0.$
\end{proof}

We are now ready to prove \Cref{thm:tree_less_star_2} in terms of a graphon system. 
Indeed, we prove the following stronger statement: 
\begin{lemma}\label{lem:induction-leafstar}
    Let $\mathbf{W} = \mathrm{span}(W_1, \dots, W_k)$ be a graphon system.
    Let $T$ be a tree with $k$ edges and $L = \{e_{1}, \dots, e_{\ell}\}$ be the set of all edges incident to leaves of $T$ except the edges of some leaf-star $S$ of $T$. Let $C \subseteq [k]$ be a subset of size $\ell.$ Let $\psi$ be a rainbow pre-coloring of $T$ such that $\mathbf{dom}(\psi) = L$ and $\psi(L) = C.$ If $\sum_{i\in [k]} \gamma(W_i) < 1$, then $t^{*}_{(T, \psi)}(\mathbf{W}) > 0.$
\end{lemma}
\begin{proof}
    We will use induction on $k.$ Without loss of generality, let $C = [\ell]$ and let $\psi(e_i) = i$ for each $i \in [\ell].$
    By \Cref{lem:star-isolated}, we may assume that $T$ is not a star. Then by \Cref{obs:twoleafstar}, there is another leaf-star $S'$ of $T$ distinct from $S.$ Let $S'$ be an $(u,v)$-leaf star.
    Suppose that $S'$ has size $k-2.$
    Then \Cref{lem:P3-second-leaf-star} applied to $T$ and $S'$ yields $t_{(T,\psi)}^{\ast}(\mathbf{W}) > 0$.
    Hence we may assume that the size of $S'$ is not equal to $k-2.$
    Since there are two distinct leaf-stars $S$ and $S'$ of $T$, we have $|L| \leq k-3$ by \Cref{obs:structure-leafstar}.

    Suppose on the contrary that there is a graphon system $\mathbf{W} = \mathrm{span}(W_1, \dots, W_k)$ such that $t^{*}_{(T, \psi)}(\mathbf{W}) = 0$ but $\sum_{i\in [k]}\gamma(W_i) < 1.$
    Without loss of generality, let $E(S') = \{e_{1}, \dots, e_m\}$. Let $\alpha := \sum_{i\in [m]}\gamma(W_i) < 1.$
    Since we have at least three colors not in $C$, we may assume that $\beta:= \gamma(W_{k}) < \frac{1 -\alpha}{3}.$
    
    Let $\psi^{\ast}$ be a rainbow pre-coloring of $T$ extending $\psi$ such that $\mathbf{dom(\psi^{\ast})} = L\cup \{uv\}$ and $\psi^{\ast}(uv) = k.$ Then $t^{*}_{(T, \psi^{\ast})}(\mathbf{W}) \leq t^{*}_{(T, \psi)}(\mathbf{W}) = 0$, so $t^{*}_{(T, \psi^{\ast})}(\mathbf{W}) = 0.$
    By applying \Cref{lem:induction-structure} with $\psi = \psi^{\ast}$, one obtains a partition of the interval $[0, 1]$ into four parts $A, B, C, D$ and a step graphon $W'_k$ as described in \Cref{lem:induction-structure}.
    Because $W_k \leq W'_k$ almost everywhere, we have
     \[
      \sum_{i\in [k-1]}\gamma(W_i) + \gamma(W'_k) \leq \sum_{i\in [k]}\gamma(W_i) < 1,
     \]
    and $\gamma(W'_k) \leq \gamma(W_k) = \beta.$

    Let $a = \mu(A)$, $b = \mu(B)$, $c = \mu(C)$ and $d = \mu(D)$. Since
    $$1 - \alpha \leq \mu(A\cup B) = \mu\left( \bigcap_{i\in [m]}\mathrm{supp}(d_{W_i}) \right),$$
    we infer that
    \begin{equation}\label{eq:1-a}
        a+b \geq 1 - \alpha.
    \end{equation}
    By the definition of $W'_k$,
    $$\gamma(W'_k) = 1 - \sqrt{(b+d)^2 + 2d(a+c) + c^2} \leq \gamma(W_k) = \beta.$$
    Consequently, we have 
    \begin{equation}\label{eq:1-b}
        (b+d)^2 + 2d(a+c) + c^2 \geq (1-\beta)^2.
    \end{equation}

    \begin{claim}\label{clm:b+d}
        $b+d \geq 1 - \alpha - \beta$.
    \end{claim}

    \begin{claimproof}
        Suppose $b+d < 1-\alpha-\beta$.
        Let $f: \mathbb{R}^4 \to \mathbb{R}$ be the function defined by
        $$f(x_1, x_2, x_3, x_4) = (x_2 + x_4)^2 + 2x_4(x_1+x_3) + x_3^2.$$
        By Equations \eqref{eq:1-a} and \eqref{eq:1-b}, there exists a tuple $(a, b, c, d)$ such that
        \begin{enumerate}[(i)]
            \item $a+b+c+d = 1$ and $a, b, c, d\geq 0$,
            \item $a+b = 1 - \alpha$,
            \item $b + d = \mu(B) + \mu(D) < 1 - \alpha - \beta$,
            \item $f(a, b, c, d) \geq (1 - \beta)^2$,
            \item $(a,b,c,d)$ maximizes $f$ among all such tuples $(a, b, c, d)$.
        \end{enumerate}
        For (ii), putting $a' = a-t$, $b' = b-t'$, $c' = c+t$ and $d' = d+t'$ for some $t,t' \geq 0$, the tuple $(a',b',c',d')$ satisfies the conditions (i) and (iii).
        Considering
         \[
          f(a',b',c',d') - f(a,b,c,d) = 2t'(a+c)+2tc+t^2 \geq 0,
         \]
        we may decrease $a+b$ to obtain (ii).

        First, assume that $b$ and $c$ are both non-zero. Let $t := \min\{b, c\}.$ Putting $a' = a + t$, $b' = b - t$, $c' = c - t$ and $d' = d + t$, the tuple $(a', b', c', d')$ satisfies (i)-(iii). Since
         \[
          f(a', b', c', d') - f(a, b, c, d) = 2ta + t^2 > 0,
         \]
        this contradicts the maximality.
        Hence at least one of $b$ and $c$ is zero.
        
        \begin{case}
            $b = 0$.
        \end{case}

        In this case, we have
         \[
          (1 - \beta)^2 \leq f(a, b, c, d) = d^2 + 2d(a+c) + c^2 = \alpha^2 + 2d(1 - \alpha).
         \]
        Then we have $1 - \alpha - \beta > d \geq \frac{(1 - \beta)^2 - \alpha^2}{2(1 - \alpha)}$ and $d \leq \alpha$. From these, we infer
        \begin{equation}\label{eq:first}
            (1 - \beta)^2 \leq 2\alpha - \alpha^2,
        \end{equation}
        \begin{equation}\label{eq:second}
            (1 - \beta)^2 < 3\alpha^2 - 4\alpha + 2\alpha\beta - 2\beta + 2.
        \end{equation}
        By reformulating \eqref{eq:first}, we have $(1 - \alpha)^2 + (1 - \beta)^2 \leq 1.$ As $\beta < \frac{1 - \alpha}{3}$, the inequality $\left(1 - \alpha\right)^2 + \left(\frac{2 + \alpha}{3} \right)^2 \leq 1$ holds, which yields $\frac{2}{5} \leq \alpha \leq 1.$
        A reformulation of \eqref{eq:second} is $\beta(\beta - 2\alpha) < (3\alpha-1)(\alpha - 1).$
        Since $\beta < \frac{1 - \alpha}{3} \leq 2\alpha$, the left-hand side decreases as $\beta$ increases.
        For $\beta < \frac{1-\alpha}{3}$, the inequality $\frac{(\alpha - 1)(7\alpha - 1)}{9} < (3\alpha - 1)(\alpha - 1)$ holds.
        This yields $\alpha < \frac{2}{5}$, whence a contradiction.

        \begin{case}
            $c = 0.$
        \end{case}

        In this case, (i) and (ii) implies $d = \alpha$.
        Hence (iii) yields $0 \leq b < 1 - \alpha - \beta - d = 1 - 2\alpha - \beta.$ Thus
        \begin{equation}\label{eq:third}
            2\alpha + \beta < 1.
        \end{equation}
        Moreover, since
         \[
          f(a, b, c, d) = (\alpha + b)^2 + 2\alpha a = (\alpha + b)^2 + 2\alpha(1 - \alpha - b) = 1 - (1 - \alpha)^2 + b^2 \geq (1 - \beta)^2,
         \]
        we get
        \begin{equation}\label{eq:fourth}
            (1 - \alpha)^2 + (1 - \beta)^2 - 1\leq b^2 < (1 - 2\alpha - \beta)^2.
        \end{equation}
        By reformulating \eqref{eq:fourth}, we have $3\alpha^2 - 2\alpha + 4\alpha\beta > 0$, which gives $\alpha > 0$ and $3\alpha + 4\beta > 2$.
        Because $\beta < \frac{1 - \alpha}{3}$, by the equation \eqref{eq:third}, we have $3\alpha + 4\beta < 2$, whence a contradiction.
    \end{claimproof}

        Consider a graphon $W_{k+1} = \mathbbm{1}_{(B\cup D)^2}.$
        By \Cref{clm:b+d}, we have $\gamma(W_{k+1}) \leq \alpha + \beta.$
        Let
         \[
          \mathbf{W}' = \mathrm{span}(W_{m+1}, \dots, W_{k-1}, W_{k+1})
         \]
        be a graphon system.
        Then
         \[
          \sum_{m+1 \leq i\leq k-1}\gamma(W_i) + \gamma(W_{k+1}) < 1.
         \]
        Let $T' := T - (S'-v).$
        Let $\psi'$ be a rainbow pre-coloring of $T'$ such that $\mathbf{dom}(\psi') = (L\cap E(T')) \cup \{uv\}$ and $\psi'(e) = \psi(e)$ for every $e\in L\cap E(T')$ and $\psi'(uv) = k+1.$ By induction hypothesis,
        $$t^{*}_{(T', \psi')}(\mathbf{W}') = \int_{[0,1]} d_{W_{k+1}}(x)rt^{*}_{(T'-v, u, \psi'|_{T'-v})}(\mathbf{W}',x) dx>0.$$
        Thus $d_{W_{k+1}}(x)rt^{*}_{(T'-v, u, \psi'|_{T'-v})}(\mathbf{W}',x)>0$ on a set of positive measure. 
        Let $S''$ be the subtree of $T$ consisting of $v$ and all of its neighborhoods and $\psi''$ be the restriction of $\psi^{\ast}$ on $E(S'')$.
        Then
        $$rt^{*}_{(S'', u, \psi'')}(\mathrm{span}(W_1, \ldots, W_{k-1}, W'_k), x) = \int_{[0, 1]} \prod_{i=1}^m d_{W_i}(y) W'_k(y, x) dy.$$
        Recall that $B \subseteq \mathrm{supp}(\prod_{i=1}^m d_{W_i})$.
        Thus by \Cref{lem:induction-structure}, we have $\prod_{i=1}^m d_{W_i}(y) W'_k(y, x)>0$ whenever $x, y \in B$. 
        Hence $rt^{*}_{(S'', u, \psi'')}(x)>0$ almost everywhere on $B$.
        Similarly, as $A \cup B = \mathrm{supp}(\prod_{i=1}^m d_{W_i})$, we have $\prod_{i=1}^m d_{W_i}(y) W'_k(y, x)>0$ whenever $x \in D$ and $y \in A \cup B$. 
        The condition $\sum_{i \in [m]} \gamma(W_i)<1$ implies that the measure of $A \cup B$ is strictly positive, so $rt^{*}_{(S'', u, \psi'')}(x)>0$ almost everywhere on $D$.
        From these, we infer that
        $$rt^{*}_{(S'', u, \psi'')}(x)rt^{*}_{(T'-v, u, \psi'|_{T'-v})}(x)>0,$$
        on a set of positive measure. 
        Therefore, $t^{*}_{(T, \psi^{\ast})}(\mathrm{span}(W_1, \ldots, W_{k-1}, W'_k))> 0$, a contradiction.
        This completes the proof.
\end{proof}

\begin{proof}[Proof of \Cref{thm:tree_less_star}]
For a nonnegative reals $\alpha_1, \ldots, \alpha_k \in [0, 1]$, if $\min_{i \in [k]} \alpha_i > \left(\frac{k-1}{k}\right)^2$ for each $i \in [k]$, then $\sum_{i=1}^k (1-\sqrt{\alpha_i}) < 1$.
So by \Cref{thm:tree_less_star_2}, we obtain $\mathrm{ex}_k^*(T) \leq \left(\frac{k-1}{k}\right)^2$ for every $k$-edge tree $T$.
By \Cref{lem:star-isolated}, $\pi_k^*(K_{1, k}) = (\frac{k-1}{k})^2$ so it completes the proof.
\end{proof}

\section*{Acknowledgement} SI, JK, and HL are supported by the National Research Foundation of Korea (NRF) grant funded by the Korea government(MSIT) No. RS-2023-00210430. 
SI and HL are supported by the Institute for Basic Science (IBS-R029-C4).
HS is supported by Institute for Basic Science (IBS-R032-D1).


\appendix

\section{Proof of the regularity lemma and \Cref{thm:compact_more}}\label{appendix:easy_proofs}
Throughout the appendix, we again assume that all the subsets of $\mathbb{R}^n$ and the functions between them that we deal with are measurable.
We first remind the simple fact that the stepping operator is a contraction in the cut-norm.
\begin{lemma}[\cite{Lovasz2012}, Exercise~9.17]\label{lem:contraction}
    Let $W:[0, 1]^2 \rightarrow \mathbb{R}$ be a bounded symmetric function and $\mathcal{P}$ be a finite partition of $[0, 1]$ into sets.
    Then $\lVert W_{\mathcal{P}} \rVert_{\square} \leq \lVert W \rVert_{\square}$.
\end{lemma}

\begin{corollary}
     Let $W:[0, 1]^2 \rightarrow \mathbb{R}$ be a bounded symmetric function, $\mathcal{P}$ a finite partition of $[0, 1]$ into sets, and $\mathcal{Q}$ a refinement of $\mathcal{P}$.
    Then $\lVert W_{\mathcal{Q}} \rVert_{\square} \leq \lVert W_{\mathcal{P}} \rVert_{\square}$.
\end{corollary}
Together with the definition of the cut-norm of graphon systems (\Cref{def:cut_norm}), the above corollary straightforwardly generalizes to graphon systems.

\begin{lemma} [\cite{Lovasz2012}, Weak Regularity lemma]\label{lem:weak_reg}
    Let $W:[0, 1]^2 \rightarrow \mathbb{R}$ be a bounded symmetric function and $k \geq 1$ be an integer.
    Then there is a step function $U$ with at most $k$ steps such that 
    $$d_{\square} (W,U) \leq \frac{2}{\sqrt{\log k}} \lVert W \rVert_2.$$
\end{lemma}

In fact, this lemma  together with the following lemma implies that there exists a partition $\mathcal{P}$ with at most $k$ parts with $d_{\square} (W,W_{\mathcal{P}}) \leq \frac{4}{\sqrt{\log k}} \lVert W \rVert_2$.
\begin{lemma}\label{lem:step_function}
    Let $\mathbf{W}=(W_1, \ldots, W_k)$ and $\mathbf{U}=(U_1, \ldots, U_k)$ be tuples of symmetric bounded functions. If each $U_i$ is a step function with steps in $\mathcal{P}$, then $d_{\square} (\mathbf{W},  \mathbf{W}_{\mathcal{P}}) \leq 2d_{\square}(\mathbf{W}, \mathbf{U})$.
\end{lemma}
\begin{proof}
    Since $\mathbf{U}=\mathbf{U}_{\mathcal{P}}$, we have
    \[
        d_{\square} (\mathbf{W},\mathbf{W}_{\mathcal{P}}) \leq d_{\square} (\mathbf{W},\mathbf{U}) +  d_{\square} (\mathbf{U},\mathbf{W}_{\mathcal{P}}) = d_{\square} (\mathbf{W},\mathbf{U}) +  d_{\square} (\mathbf{U}_{\mathcal{P}},\mathbf{W}_{\mathcal{P}}) \leq 2 d_{\square} (\mathbf{W},\mathbf{U}). \qedhere
    \]
\end{proof}

We now generalize the weak regularity lemma to a tuple of graphons.
\begin{lemma}\label{lem:graphon_regular}
    $m_1, m_2 \in \mathbb{Z}_{>0}$ be integers with $m_2 \geq k m_1$. Let $\mathbf{W} = (W_1, \ldots, W_k)$ be a tuple of bounded symmetric functions and $\mathcal{Q}$ be a partition of $[0, 1]$ into $m_1$ parts. 
    Then there exists a refinement $\mathcal{P}$ of $\mathcal{Q}$ with at most $m_2$ parts such that 
    $$d_{\square} (\mathbf{W}, \mathbf{W}_{\mathcal{P}}) \leq k^{1/2} \frac{8 \sum_{i \in [k]} \lVert W_i \rVert_2}{\sqrt{\log (m_2/m_1)}}.$$
\end{lemma}
\begin{proof}
    We choose the largest integer $t$ such that $t^{k} \leq m_2/m_1$. From the choice, we have 
    $$\log t \leq k^{-1} \log(m_2/m_1) \leq(t+1)$$ and $t \geq 2$.
    We now apply Lemma~\ref{lem:weak_reg} and Lemma~\ref{lem:step_function} for each $W_i$.
    Then we obtain a partition $\mathcal{P}_i$ with at most $t$ parts for each $i \in [k]$ such that
     \[
      d_{\square} (W_i, (W_i)_{\mathcal{P}_i}) \leq \frac{4}{\sqrt{\log t}} \lVert W_i \rVert_2.
     \]
    Let $\mathcal{P}$ be the common refinement of all such partitions $\mathcal{P}_i$'s and $\mathcal{Q}$, then it has at most $m_1 t^k \leq m_2$ parts.
    Then by Lemma~\ref{lem:contraction}, 
    $$d_{\square} (\mathbf{W}, \mathbf{W}_{\mathcal{P}})
    \leq \sum_{i \in [k]} d_{\square} (W_i, (W_i)_{\mathcal{P}})
    \leq  \frac{4 \sum_{i \in [k]} \lVert W_i \rVert_2}{\sqrt{\log t}}
    \leq k^{1/2} \frac{8 \sum_{i \in [k]} \lVert W_i \rVert_2}{\sqrt{k \log (t+1)}}
    \leq k^{1/2} \frac{8 \sum_{i \in [k]} \lVert W_i \rVert_2}{\sqrt{\log (m_2/m_1)}}. \qedhere$$
\end{proof}
Note that following the proof of Lemma~\ref{lem:weak_reg} leads to a better bound in Lemma~\ref{lem:graphon_regular}.
However, Lemma~\ref{lem:graphon_regular} is already adequate for establishing our main results.
 To prove Theorem~\ref{thm:compact_more}, we also need basic facts about martingale.
\begin{definition}
    A sequence of random variables $(X_n)_{n \geq 0}$ is called a \textit{martingale} if $\mathbb{E}[X_t \mid X_0, \ldots, X_{t-1}] = X_{t-1}$ for every $t \geq 1$.
\end{definition}
\begin{theorem}[\cite{Durrett}, Martingale convergence theorem]\label{thm:martingale_convergence}
    Let $(X_n)_{n \geq 0}$ be a martingale.
    If $\sup_n \mathbb{E}[|X_n|] < \infty$, then $(X_n)$ converges almost everywhere.
\end{theorem}

We are now ready to prove Theorem~\ref{thm:compact_more}. 
\begin{proof}[Proof of~\Cref{thm:compact_more}]
    It suffices to prove that $X$ is sequentially compact.
    Let $(\mathbf{W}^n)$ be a sequence of graphon systems in $X$, and let $W_{I}^i$ be the component of $\mathbf{W}^i$ corresponding to the subset $I \subseteq [k]$.
    
    For each $n$, we apply Lemma~\ref{lem:graphon_regular} iteratively to obtain partitions $\mathcal{P}_{1}^n = \{[0, 1]\}, \mathcal{P}_{2}^n, \ldots$ such that 
    \begin{enumerate}
        \item[(a)] $\mathcal{P}_{t+1}^n$ is a refinement of $\mathcal{P}_{t}^n$;
        \item[(b)] $d_{\square} (\mathbf{W}^n, (\mathbf{W}^n)_{\mathcal{P}_{t}^n}) < \frac{1}{t}$;
        \item[(c)] $\mathcal{P}^{n}_{t}$ has exactly $m_t$ parts, while this number $m_t$ only depends on $t$.
    \end{enumerate}
    By allowing an empty part in each $\mathcal{P}^{n}_{t}$, one can choose such partitions.

    Let $\mathbf{W}_{t}^n = (\mathbf{W}^n)_{\mathcal{P}_{t}^n} \in X$ and $W_{ t, I}^n$ be the component of $\mathbf{W}_{t}^n$ corresponding to $I \subseteq [k]$.
    By applying appropriate measure-preserving bijections to both $\mathcal{P}_{t}^n$ and $\mathbf{W}_{t}^n$ for each $n$ and ignoring sets of measure zero, we may assume that all the partitions $\mathcal{P}_t^n$ consist of intervals.
    Note that after replacing, the equality $(\mathbf{W}_{t}^n)_{\mathcal{P}_{t'}^n} = \mathbf{W}_{t'}^n$ still holds for $t'<t$.
    
    We claim that there exists a subsequence $(n_i)$ such that for any $t$, the sequence $(\mathbf{W}_{t}^{n_i})$ converges to some $\mathbf{U}_t$ almost everywhere.
    By the standard diagonalization argument, it suffices to show that for each fixed $t$, there exists a subsequence $(\mathbf{W}^{n_i})$ such that $(\mathbf{W}_{t}^{n_i})$ converges to some $\mathbf{U}_t$ almost everywhere.
    Fix $t$.
    Let $\ell_n(i)$ be the length of the $i$-th interval of $\mathcal{P}_{t}^n$.
    Let $f_{n}(i, j, I)$ be the value of $W_{t, I}^n$ on the square ($i$-th interval of $\mathcal{P}_{t}^n) \times (j$-th interval of $\mathcal{P}_{t}^n)$.
    If one of those intervals is empty, then take any value in $[0, 1]$.
    By compactness of $[0, 1]^{m_t} \times ([0, 1]^{m_t^2})^{2^k}$, there exists a subsequence $(a_s)$ such that $\ell_{a_s}(i)$ and $f_{a_s}(i, j, I)$ converge for every $i, j \in [m_k]$ and $I \subseteq [k]$.
    Then each sequence $(W_{t, I}^{a_s})_s$ converges to some step function $U_{t,I}$ almost everywhere and for each fixed $t$, so $(\mathbf{W}_{t}^{a_s})_s$ converges to some $\mathbf{U}_t$ almost everywhere.
    By passing to a subsequence, we may assume that for each fixed $t$, the sequence $(\mathbf{W}_{t}^n)$ converges to $\mathbf{U}_t$ almost everywhere.
    
    The pointwise limit $\mathbf{U}_t$ is a graphon system in $X$ by the second condition of Theorem~\ref{thm:compact_more}.
    Also, the limit of $\ell_{n_a}(i)$ defines a partition $\mathcal{P}_t$ consisting of steps of $\mathbf{U}_t$. As $(\mathbf{W}_{ t}^n)_{\mathcal{P}_{t'}^n} = \mathbf{W}_{t'}^n$ for $t'<t$, we infer that $(\mathbf{U}_{t})_{\mathcal{P}_{t'}} \equiv \mathbf{U}_{t'}$ for $t'<t$.

    We claim that the sequence $(\mathbf{U}_t)$ converges to a graphon system $\mathbf{U}$ almost everywhere.
    Fix $I \subseteq [k]$. Choose $(x, y) \in [0, 1]^2$ uniformly at random and consider a sequence $(U_{t, I}(x, y))_t$.
    Since $(\mathbf{U}_{t})_{\mathcal{P}_{t'}} \equiv \mathbf{U}_{t'}$ for $t'<t$, the sequence $(U_{t, I}(x, y))_t$ is a martingale.
    Clearly, its expectation is bounded by $1$, so the sequence $(U_{t, I}(x, y))_t$ converges almost everywhere by the martingale convergence theorem~\ref{thm:martingale_convergence}.
    Therefore, $(\mathbf{U}_t)$ converges to some $\mathbf{U}$ almost everywhere.
    Again, each $\mathbf{U}_t$ is in $X$, so is $\mathbf{U}$.
    By the bounded convergence theorem, the pointwise convergence implies the convergence in $L^1$, and thus the convergence in $\delta_\square$-norm. 

   Note that
    \[
     \delta_{\square}(\mathbf{W}^{n_i}, \mathbf{U}) \leq \delta_{\square}(\mathbf{W}^{n_i}, \mathbf{W}_{t}^{n_i}) + \delta_{\square}(\mathbf{W}_{t}^{n_i}, \mathbf{U}_{t}) + \delta_{\square}(\mathbf{U}_t, \mathbf{U}).
    \]
From this, we conclude that for any given $\varepsilon > 0$, there exist sufficiently large $t$ and $i$ such that $\delta_{\square}(\mathbf{W}^{n_i}, \mathbf{U})<\varepsilon$ by the condition (b) and the above claims.
    Therefore, $(\mathbf{W}^n)$ converges to a graphon system $\mathbf{U} \in X$, which completes the proof.
\end{proof}


\section{Proof of \Cref{thm:W_random_graph}}\label{sec:W_random_graph}
We need several lemmas below for the proof of \Cref{thm:W_random_graph}.
\begin{lemma}[\cite{Lovasz2012}, Lemma~10.6, First sampling lemma]\label{lem:first_sampling}
    Let $W:[0,1]^2 \rightarrow [-1,1]$ be a symmetric function. 
    Let $S$ be a set of random $n$ points in $[0, 1]$.
    Then with probability at least $1-4\exp(-\sqrt{n}/10)$, we have $$- \frac{3}{n} \leq \lVert W[S] \rVert_{\square}  - \lVert W \rVert_{\square} \leq \frac{8}{n^{1/4}}.$$
\end{lemma}
\begin{lemma}[\cite{Lovasz2012}, Lemma~10.11]\label{lem:H_and_W}
    For any $\varepsilon>10/\sqrt{n}$ and any $n$-vertex weighted graph $H$ with edge weights in $[0,1]$, 
    $$\mathbb{P}(d_{\square} (\mathbb{G}(H), H) >\varepsilon)<\exp\,\left( -\frac{\varepsilon^2 n^2}{100} \right),$$
\end{lemma}
\begin{lemma}[Generalization of \cite{Lovasz2012}, Lemma~10.3]\label{lem:concentration}
    Let $f$ be a function from the set of tuples $\mathbf{H}=(H_I)_{\emptyset \subsetneq I \subseteq [k]}$ of weighted graphs to $\mathbb{R}$. 
    Suppose that $|f(\mathbf{H})-f(\mathbf{H}')| \leq 1$ whenever there exists a vertex $v$ such that for each $\varnothing \neq I \subseteq [k]$, $H_I$ and $H'_I$ differ only in edges incident to $v$.
    Then for every $t \geq 0$, we have $$\mathbb{P}(f(\mathbb{H}(n, \mathbf{W})) \geq \mathbb{E}[f(\mathbb{H}(n, \mathbf{W}))] + \sqrt{2tn}) \leq e^{-t}.$$
    When $f$ is defined on the set of graph systems of order $k$ (by letting $G_I = \cap_{i \in I} G_i$),
    $$\mathbb{P}(f(\mathbb{G}(n, \mathbf{W})) \geq \mathbb{E}[f(\mathbb{G}(n, \mathbf{W}))] + \sqrt{2tn}) \leq e^{-t}.$$
\end{lemma}
A partition $\mathcal{P}$ of $[0,1]$ is said to be \emph{equitable} if all the parts have the same measure.
\begin{restatable}
{lemma}{equitable}\label{lem:equitable_partition}
    Let $\mathbf{W} = (W_I)_{I \subseteq [k]}$ be a graphon system of order $k$.
    For any integer $m \in \mathbb{Z}_{> 1}$, there exists an equitable partition $\mathcal{P}$ of $[0, 1]$ into $m$ parts such that
    $$d_{\square} (\mathbf{W}, \mathbf{W}_{\mathcal{P}}) \leq 2^k \left( \frac{50 \times 4^k}{\sqrt{\log m}} + \frac{2k}{\sqrt{m}} \right).$$
\end{restatable}

The proof of \Cref{lem:concentration} is the standard application of Azuma's inequality, so we omit the proof. The following lemma will be useful to prove  \Cref{lem:equitable_partition}. 
\begin{lemma}\label{lem:equitable_partition_2}
    Let $\mathbf{W} = (W_1, \ldots, W_k)$ be a tuple of bounded symmetric functions and $\mathcal{Q}$ be a partition of $[0, 1]$ into $m$ parts. 
    Then for given $m'>m$, there exists an equitable partition $\mathcal{P}$ into $m'$ parts such that 
    $$d_{\square}(\mathbf{W}, \mathbf{W}_{\mathcal{P}}) \leq 2d_{\square}(\mathbf{W}, \mathbf{W}_{\mathcal{Q}}) + \frac{2km}{m'}.$$
\end{lemma}
\begin{proof}
     We partition each part of $\mathcal{Q}$ into parts of measure $\frac{1}{m'}$ except at most one part.
     We collect all the exceptional parts and partition their union into parts of measure $\frac{1}{m'}$.
     Let $\mathcal{P}$ be the resulting partition.
     Let $\mathcal{R}$ be the common refinement of $\mathcal{P}$ and $\mathcal{Q}$.
     Then $\mathbf{W}_{\mathcal{R}}$ and $\mathbf{W}_{\mathcal{P}}$ differ only in the exceptional parts and there are at most $m$ exceptional parts.
     Hence we have $d_{\square} (\mathbf{W}_{\mathcal{R}}, \mathbf{W}_{\mathcal{P}}) \leq 2k\frac{m}{m'}$ and $d_{\square} (\mathbf{W}_{\mathcal{P}}, \mathbf{W}) \leq d_{\square} (\mathbf{W}_{\mathcal{R}}, \mathbf{W}) + \frac{2km}{m'}$.
     From Lemma~\ref{lem:step_function}, we have $d_{\square} (\mathbf{W}_{\mathcal{R}}, \mathbf{W}) \leq 2d_{\square} (\mathbf{W}_{\mathcal{Q}}, \mathbf{W})$, which concludes the proof.
\end{proof}
\begin{proof}[Proof of \Cref{lem:equitable_partition}]
    Let $m' = \lfloor \sqrt{m} \rfloor$ and $\mathbf{W}=(W_I)_{I \subseteq [k]}$ be a graphon system of order $k$.
    We first apply the Lemma~\ref{lem:graphon_regular} with $m_1 = 1$ and $m_2 = m'$ to get a partition $\mathcal{Q}$ of $[0,1]$ into at most $m'$ parts with $d_{\square} (\mathbf{W},  \mathbf{W}_{\mathcal{Q}}) \leq 2^{5k/2} \frac{8}{\sqrt{\log m'}}$. 
    Then by applying Lemma~\ref{lem:equitable_partition_2} to $\mathbf{W}$ and $\mathcal{Q}$, we get a desired equitable partition $\mathcal{P}$ into $m$ parts.
\end{proof}

The final ingredient for the proof of \Cref{thm:W_random_graph} is the following version of the sampling lemma.
\begin{lemma}\label{lem:step_function2}
    Let $\mathcal{P}$ be an equitable partition with $m$ parts and $\mathbf{W} = (W_I)_{I \subseteq [k]}$ be a graphon system of order $k$ such that every $W_I$ is a step graphon with steps in $\mathcal{P}$. 
    Let $S$ be a tuple in $[0, 1]^n$ chosen uniformly at random.
    Then $\mathbb{E}[\delta_{\square}(\mathbf{W}, \mathbf{W}[S])] \leq \frac{2^{k+2}\sqrt{m}}{\sqrt{n}}$.
\end{lemma}
\begin{proof}
    Let $P_i$ be the $i$-th part of $\mathcal{P}$.
    Note that we can consider the weighted graph $W_I[S]$ as a step graphon by regarding vertices in $[n]$ as disjoint subsets of $[0,1]$. 
    This allows us to view both $\mathbf{W}$ and $\mathbf{W}[S]$ as step graphons with $m$ steps. 
    Each step of $\mathbf{W}$ has measure exactly $\frac{1}{m}$, while each step of $\mathbf{W}[S]$ has length $\ell_i = \frac{1}{n} |S \cap P_i|$.
    Observe that $n \cdot \ell_i$ is identical with the binomial random variable $\mathrm{Binom}(n, \frac{1}{m})$. 

    Write $S = (x_1, \dots, x_n)$ as a tuple.
    Choose a measure preserving bijection $\varphi:[0,1] \rightarrow [0,1]$ such that the parts of $\mathbf{W}^{\varphi}$ are intervals.
    Define a reordering $\pi:[n] \rightarrow [n]$ as follows.
    First, there exists a reordering $\pi':[n] \rightarrow [n]$ such that
     \[
      S \cap P_i = \left\{ x_{\pi'^{-1}(n(\ell_1 + \dots + \ell_i) + 1)}, \dots, x_{\pi'^{-1}(n(\ell_1 + \dots + \ell_i + \ell_{i+1}))} \right\}.
     \]
    Second, we define $\pi'':[n] \rightarrow [n]$ as follows: if $n\ell_i \leq \lfloor \frac{im}{n} \rfloor - \lfloor \frac{(i-1)m}{n} \rfloor$, define
     \[
      \pi''(n(\ell_1 + \dots + \ell_i) + j) = \left\lfloor \frac{(i-1)m}{n} \right\rfloor + j,
     \]
    for $1 \leq j \leq n\ell_i$; otherwise, define
     \[
      \pi''(n(\ell_1 + \dots + \ell_i) + j) = \left\lfloor \frac{(i-1)m}{n} \right\rfloor + j,
     \]
    for $1 \leq j \leq \lfloor \frac{im}{n} \rfloor - \lfloor \frac{(i-1)m}{n} \rfloor$, and send all the other remaining indices arbitrarily to the indices not mapped to yet.
    Set $\pi = \pi'' \circ \pi'$.
    Let $\varphi':[0,1] \rightarrow [0,1]$ be a measure preserving bijection such that $\mathbf{W}[S]^{\varphi'} = \mathbf{W}[S^{\pi}]$, where $S^{\pi} = (x_{\pi(1)}, \dots, x_{\pi(n)})$.
    Then the set
     \[
      \left\{ x \in P_i \times P_j: W_I^{\varphi}(x) \neq (W_I[S])^{\varphi'}(x) \right\}
     \]
    has measure at most $|\ell_i - \frac{1}{m}| + |\ell_j - \frac{1}{m}|$ for each pair $i, j \in [m]$.
    Because the random variable $\ell_i$ are independent and identically distributed, we infer that
    $$\mathbb{E}\bigr[\lVert W_{I} - (W_I[S])^{\varphi^{-1} \circ \varphi'} \rVert_{1}\bigr] \leq 4\mathbb{E} \left[ \sum_{i=1}^m \left| \ell_i - \frac{1}{m} \right| \right]
    = 4m \mathbb{E} \left[ \left| \ell_1 - \frac{1}{m} \right| \right]
    \leq 4m \sqrt{\mathbb{E} \left[ \left| \ell_1 - \frac{1}{m} \right|^2 \right]}
    = 4 \sqrt{\frac{m-1}{n}}
    \leq \frac{4\sqrt{m}}{\sqrt{n}}.$$
    As $\lVert W \rVert_\square \leq \lVert W \rVert_1$ for every bounded symmetric function $W$, this concludes the proof.
\end{proof}
\begin{proof}[Proof of \Cref{thm:W_random_graph}]
    First, we aim to bound the expectation of the distance. Then we will apply Lemma~\ref{lem:concentration}. 
    We may assume that $n$ is sufficiently large.
    Let $\mathbf{W}$ be a graphon system of order $k$.
    Let $m = \lceil n^{1/4} \rceil$.
    By Lemma~\ref{lem:equitable_partition}, there exists an equitable partition $\mathcal{P}$ of $[0, 1]$ into $m$ parts such that 
    \begin{align}\label{eq: W-WP}
        d_{\square} (\mathbf{W}, \mathbf{W}_{\mathcal{P}}) \leq \frac{200 \times 8^k}{\sqrt{\log n}}.
    \end{align}
    
    Let $S$ be a tuple in $[0,1]^n$ chosen uniformly at random. Note that $\mathbf{H}_S(n, \mathbf{W})$ and $\mathbf{W}[S]$ have the same distribution except for the weights of loops. 
    Thus, $\mathbb{E}[d_{\square} (\mathbf{H}_S(n, \mathbf{W}), \mathbf{W}[S])] \leq \frac{2^k}{n}$.
    Then by Lemma~\ref{lem:first_sampling}, with probability at least $1-4\exp(-n^{1/8}/20)$, the following holds for each $I\subseteq [k]$.
    $$\big| d_{\square} (W_I[S], (W_I)_{\mathcal{P}}[S]) - d_{\square} (W_I, (W_I)_{\mathcal{P}}) \big| \leq  \frac{8}{n^{1/4}}.$$
    Thus we have
    $$\mathbb{E}\Bigr[\big| d_{\square} (W_I[S], (W_I)_{\mathcal{P}}[S]) - d_{\square} (W_I, (W_I)_{\mathcal{P}}) \big|\Bigr] \leq \left( 1-4\exp\left(-\frac{n^{1/8}}{20} \right) \right) \cdot \frac{8}{n^{1/4}} + 4\exp\left( -\frac{n^{1/8}}{20} \right) \leq   \frac{10}{n^{1/4}}.$$
    These inequalities together with \eqref{eq: W-WP} imply that
    \begin{align*}
        \mathbb{E}\bigr[d_{\square} (\mathbf{W}[S], \mathbf{W}_{\mathcal{P}}[S]) \bigr]
        & \leq \mathbb{E}\Bigr[\big| d_{\square} (\mathbf{W}[S], \mathbf{W}_{\mathcal{P}}[S]) - d_{\square} (\mathbf{W}, \mathbf{W}_{\mathcal{P}}) \big|\Bigr] + d_{\square} (\mathbf{W}, \mathbf{W}_{\mathcal{P}}) \\
        & \leq \sum_{I \subseteq [k]}\mathbb{E}\Bigr[\big|d_{\square} (W_I[S], (W_I)_{\mathcal{P}}[S]) - d_{\square} (W_I, (W_I)_{\mathcal{P}}) \big|\Bigr] + d_{\square} (\mathbf{W}, \mathbf{W}_{\mathcal{P}}) \leq \frac{300 \times 8^k}{\sqrt{\log n}}.
    \end{align*}
    Finally, by Lemma~\ref{lem:step_function2}, we have
     \[
      \mathbb{E}[\delta_{\square}(\mathbf{W}_{\mathcal{P}}, \mathbf{W}_{\mathcal{P}}[S])] \leq \frac{2^{k+2}}{n^{3/4}} \leq \frac{100 \times 2^k}{\sqrt{\log n}}.
     \]
    Therefore,
    \begin{align*}
        \mathbb{E}\bigr[\delta_{\square}(\mathbf{W}, \mathbf{W}[S])\bigr] & \leq \delta_{\square}(\mathbf{W}, \mathbf{W}_{\mathcal{P}}) + \mathbb{E}\bigr[\delta_\square( \mathbf{W}[S] , \mathbf{W}_{\mathcal{P}}[S])\bigr] + \mathbb{E}\bigr[\delta_{\square}(\mathbf{W}_{\mathcal{P}}, \mathbf{W}_{\mathcal{P}}[S])\bigr]  \leq \frac{400 \times 8^k}{\sqrt{\log n}}.
    \end{align*}
    To estimate $\mathbb{E}[\delta_{\square}(\mathbf{W}, \mathbb{G}(n, \mathbf{W}))]$, observe that for a random sample $\mathbf{H}$ of $\mathbb{H}(n, \mathbf{W})$ and $i \neq j \in [n]$, the event that $ij$ is an edge of $G_I = \bigcap_{t \in I} G_t$ have probability exactly the weight of the edge $ij$ of $H_I$.
    Also, it is independent from the event that $i'j'$ is an edge of $G_I$ whenever $\{i', j'\} \neq \{i, j\}$.
    Thus by Lemma~\ref{lem:H_and_W} applied to each index $I \subseteq [k]$, the union bound yields that $d_{\square} (\mathbb{G}(\mathbf{H}), \mathbf{H}) \leq \frac{1}{\sqrt{\log n}}$ with probability at least $1-2^k \cdot \exp(-n^2/(100\times 2^{2k}\log n))$. We also have that $d_{\square} (\mathbb{G}(\mathbf{H}), \mathbf{H}) \leq 2^k$ for any case.
    Therefore, $$\mathbb{E}\bigr[d_{\square} (\mathbb{G}(\mathbf{H}), \mathbf{H}) \bigr] \leq \left( 1-2^k \cdot \exp \left(-\frac{n^2}{100\times 2^{2k}\log n} \right) \right) \times \frac{1}{\sqrt{\log n}} + \left( 2^k \cdot \exp \left( -\frac{n^2}{100\times 2^{2k}\log n} \right) \right) \times 2^k \leq \frac{2}{\sqrt{\log n}}.$$
    By the law of total probability, we have
     \[
      \mathbb{E}\bigr[d_{\square} (\mathbb{G}(\mathbb{H}(n, \mathbf{W})), \mathbb{H}(n, \mathbf{W}))\bigr] \leq \frac{2}{\sqrt{\log n}}.
     \] 
    Finally, the triangle inequality together with the above inequalities imply $$\mathbb{E}\bigr[\delta_{\square}(\mathbf{W}, \mathbb{G}(n, \mathbf{W}))\bigr] \leq \frac{500 \times 8^k}{\sqrt{\log n}}.$$
    One can easily check that the function $f(\mathbf{H}):=n\delta_{\square}(\mathbf{H}, \mathbf{W})/2^{k+1}$ defined on the tuple of weighted graphs satisfies the condition of Lemma~\ref{lem:concentration}. Therefore Lemma~\ref{lem:concentration} finishes the proof. 
\end{proof}
We finally note that by applying the common refinement argument in the proof of \Cref{lem:weak_reg} and \ref{lem:equitable_partition} to Lemma~9.3 of \cite{Lovasz2012}, we can prove the following version of the multi-color graph regularity lemma for graph systems.
Note that a partition $X_1, \ldots, X_m$ of a finite set $X$ is called \emph{equitable} if $||X_i|-|X_j|| \leq 1$ for every $1 \leq i<j \leq m$.
\begin{lemma}\label{lem:weak_reg_graph}
    Let $\mathcal{G}$ be a weighted graph system of order $k$ and $m$ be a positive integer.
    Then there exists an equitable partition $\mathcal{P}$ of vertices into $m$ parts such that 
    $$\delta_{\square}(\mathrm{span}(\mathcal{G}), \mathrm{span}(\mathcal{G})_{\mathcal{P}}) \leq \frac{50 \times 2^{3k}}{\sqrt{\log m}}.$$
\end{lemma}


\section{Proof of \Cref{thm:H_removal}}
\label{sec:H-removal}
 
In order to prove \Cref{thm:H_removal}, we need the following two lemmas. 
\begin{lemma}[\cite{Lovasz2012}, Lemma~8.22]\label{lem:innerproduct_convergence}
   For each $n\in \mathbb{N}$,
    let $W^n:[0, 1]^2 \rightarrow [-1, 1]$ be a  bounded symmetric function.
    If $\lVert W^n \rVert_{\square} \rightarrow 0$ as $n \rightarrow \infty$, then for every $Z \in L^1([0, 1])$, we have $\lVert Z W^n \rVert_{\square} \rightarrow 0$.
    In particular, $\int_{[0, 1]^2} Z W^n dxdy \rightarrow 0$ and $\int_{S} W^n dxdy \rightarrow 0$ for every set $S \subseteq [0, 1]^2$.
\end{lemma}

Recall that we identify a graph system $\mathcal{G}$ with the graphon system $\mathrm{span}(\mathcal{G})$.

\begin{lemma}[Generalization of Theorem~11.59 of \cite{Lovasz2012}]\label{lem:reordering}
    Let $\mathcal{G}^n$ be a sequence of graph system such that $\mathrm{span}(\mathcal{G}^n) \rightarrow \mathbf{W}$ as $n \rightarrow \infty$ for a graphon system $\mathbf{W}$ and $|V(\mathcal{G}^n)| \rightarrow \infty$. 
    Then there exists a reordering of vertices on $\mathcal{G}^n$ such that $d_{\square} (\mathrm{span}(\mathcal{G}^n), \mathbf{W}) \rightarrow 0$.
\end{lemma}
\begin{proof}
    We claim that if $(\mathcal{G}^n)$ and $(\mathcal{H}^n)$ are two sequences of systems of (weighted) graphs with $|V(\mathcal{G}^n)| = |V(\mathcal{H}^n)| \rightarrow \infty$ and $\delta_{\square}(\mathcal{G}^n, \mathrm{span}(\mathcal{H}^n)) \rightarrow 0$, there exists a reordering $\mathcal{G}'^n$ of the vertices of $\mathcal{G}^n$ such that $d_{\square} (\mathcal{G}'^n, \mathcal{H}^n ) \rightarrow 0$.
    Indeed, by \Cref{lem:step_function}, there exists a reordering $\mathcal{G}'^n$ of $\mathcal{G}^n$ such that 
    $$d_{\square} (\mathcal{G}'^n, \mathbf{W}) \leq d_{\square} (\mathcal{G}'^n, \mathbf{W}_{\mathcal{P}_n}) + d_{\square} (\mathbf{W}_{\mathcal{P}_n}, \mathbf{W}) \leq  3d_{\square} (\mathcal{G}'^n, \mathbf{W}_{\mathcal{P}_n}) \rightarrow 0,$$
    where $\mathcal{P}_n$ is the partition of $[0, 1]$ corresponding to $\mathcal{G}'^n$.

    To prove the claim, we apply \Cref{lem:weak_reg_graph} to $\mathcal{G}^n$ and $\mathcal{H}^n$ with $m=|V(\mathcal{G}^n)|^{1/3}$.
    Then we have equitable partitions $\mathcal{P}_n$ and $\mathcal{Q}_n$ of vertices into $m$ parts such that $d_{\square}(\mathcal{G}^n, (\mathcal{G}^n)_{\mathcal{P}_n}) \leq \frac{50 \times 8^k}{\sqrt{\log m}}$ and $d_{\square}(\mathcal{H}^n, (\mathcal{H}^n)_{\mathcal{Q}_n}) \leq \frac{50 \times 8^k}{\sqrt{\log m}}$.
    By the triangle inequality, 
    $$d_{\square} ((\mathcal{G}^n)_{\mathcal{P}_n}, (\mathcal{H}^n)_{\mathcal{Q}_n}) \leq d_{\square} (\mathcal{G}^n, \mathcal{H}^n) + \frac{100 \times 8^k}{\sqrt{\log m}}.$$
    Choose a measure preserving bijection $\varphi_n:[0, 1] \rightarrow [0, 1]$ such that
    $$ d_{\square} (((\mathcal{G}^n)_{\mathcal{P}_n})^{\varphi_n}, (\mathcal{H}^n)_{\mathcal{Q}_n}) \leq \delta_{\square}((\mathcal{G}^n)_{\mathcal{P}_n}, (\mathcal{H}^n)_{\mathcal{Q}_n}) + \frac{1}{\sqrt{\log m}}. $$
    Let $P_i$ be the subset of $[0, 1]$ that corresponds to the $i$-th part of $\mathcal{P}_n$, and similarly define $Q_i$ for $i \in [m]$.
    For each $i, j \in [m]$, let $X_{i, j}$ be the measure of $\varphi_n(P_i) \cap Q_j$.
    We now define a reordering of the vertices of $\mathcal{G}^n$ and $\mathcal{H}^n$ as follows.
    For each $i, j \in [m]$, associate vertices of $j$-th part of $\mathcal{Q}_n$ one-to-one to $\lfloor X_{ij} |V(\mathcal{G}^n)|\rfloor$ vertices of the $i$-th part of $\mathcal{P}_n$.
    For the remaining vertices, we arbitrarily associate them to the remaining unmatched vertices of $\mathcal{Q}_n$.
    It is possible as $\sum_{j \in [m]} X_{ij} |V(\mathcal{G}^n)|$ is the size of $i$-th part $P_i$.
    Let $\varphi'_n$ be a measure-preserving map that corresponds to this vertex reordering.
    Then as the difference between $\lfloor X_{ij} |V(\mathcal{G}^n)|\rfloor/|V(\mathcal{G}^n)|$ and $X_{ij}$ is at most $1/|V(\mathcal{G}^n)|$ for each $i, j \in [m]$, we have 
    $$\lVert((\mathcal{G}^n)_{\mathcal{P}_n}^{\varphi'_n})_I - ((\mathcal{G}^n)_{\mathcal{P}_n}^{\varphi_n})_I \rVert_1 \leq \frac{m^2}{|V(\mathcal{G}^n)|},$$
    for every $I \subseteq [k]$.
    Therefore, we have 
    \begin{align*}
        d_{\square} ((\mathcal{G}^n)^{\varphi'_n}, \mathrm{span}(\mathcal{H}^n))
        & \leq d_{\square} (((\mathcal{G}^n)_{\mathcal{P}_n})^{\varphi'_n},  (\mathcal{H}^n)_{\mathcal{Q}_n}) + \frac{100 \times 8^k}{\sqrt{\log m}}\\
        & \leq d_{\square} (((\mathcal{G}^n)_{\mathcal{P}_n})^{\varphi_n}, (\mathcal{H}^n)_{\mathcal{Q}_n}) \\
        &\qquad + d_{\square} (((\mathcal{G}^n)_{\mathcal{P}_n})^{\varphi'_n}, ((\mathcal{G}^n)_{\mathcal{P}_n})^{\varphi_n}) + \frac{100 \times 8^k}{\sqrt{\log m}} \\
        & \leq \delta_{\square}((\mathcal{G}^n)_{\mathcal{P}_n}, (\mathcal{G}^n)_{\mathcal{Q}_n}) + \frac{200 \times 8^k}{\sqrt{\log m}} \rightarrow 0,
    \end{align*}
    as $n \rightarrow \infty$.
\end{proof}

\begin{proof}[Proof of \Cref{thm:H_removal}]
    We may assume that $(H,\psi)$ is a coloring tuple, $H$ has no isolated vertices, and $|E(H)| \geq 2$.
    When $H$ is a graph with two vertices and parallel edges, then $\varepsilon = \delta$ works by deleting one edge from each rainbow copy of $(H, \psi)$.
    So we can assume that $t = |V(H)| \geq 3$.
    
    Suppose that the statement is false.
    There exists $\varepsilon>0$ such that for every $i>0$, there is a graph system $\mathcal{G}^i = (G_{1}^i, \ldots, G_{k}^i)$ with at most $\frac{1}{i} |V(\mathcal{G}^i)|^{t}$ rainbow copies of $(H, \psi)$ but cannot be made rainbow $(H, \psi)$-free by deleting at most $\varepsilon |V(\mathcal{G}^i)|^2$ edges.
    Let $n_i = |V(\mathcal{G}^i)|$. Since $\frac{1}{i} n_i^{t} \geq \varepsilon n_i^2$ and $t \geq 3$, we have $n_i \rightarrow \infty$ as $i \rightarrow \infty$.
    As $t^{*}_{(H, \psi)}(\mathcal{G}^i) \leq \frac{1}{i} + O(\frac{1}{n_i})$, we have $t^{*}_{(H, \psi)}(\mathcal{G}^i) \rightarrow 0$ as $i \rightarrow \infty$.
    
    By passing to a subsequence, we may assume that $(\mathcal{G}^i)_{i \geq 1}$ converges to an admissible graphon system $\mathbf{W}$ as $i \rightarrow \infty$.
    Then $t^{*}_{(H, \psi)}(\mathbf{W}) = 0$ by continuity.
    By \Cref{lem:reordering}, we may assume that $d_{\square} (\mathcal{G}^i, \mathbf{W}) \rightarrow 0$. Let $S_I = \{ (x, y) \in [0, 1]^2 \mid W_I(x,y)>0\}$.
    Then by Lemma~\ref{lem:innerproduct_convergence}, for each $I \subseteq [k]$, we have
    $$
    \lim_{i \to \infty} \int_{[0, 1]^2} (1-\mathbbm{1}_{S_I})G_{I}^i dxdy = \int_{[0, 1]^2} (1-\mathbbm{1}_{S_I}) W_I dxdy =0,
    $$
    where $G_{I}^i = \prod_{j \in I} G_j^i$ for $I \neq \emptyset$ and $G_{\emptyset}^i \equiv 1$.
    Therefore, we can choose an index $i$ such that
    $$
    \int_{[0, 1]^2} (1-\mathbbm{1}_{S_I})G_{I}^i dxdy = \int_{[0, 1]^2} (1-\mathbbm{1}_{S_I})G_{I}^i dxdy < \frac{\varepsilon}{2^{k+2}|E(H)|},
    $$
    for every $I \subseteq [k]$.

    Now, we delete edges from $\mathcal{G}^i$ to make it rainbow $(H,\psi)$-free.
    For each vertex $v$ of $\mathcal{G}^i$, let $J_v \subset [0, 1]$ be the interval corresponding to $v$.
    For each $I \subseteq [k]$, if there exists an edge $uv \in \bigcap_{j \in I} G_{j}^i$ such that
     \[
      \mu(S_I \cap (J_u \times J_v)) < \frac{1}{n_i^2} \left( 1-\frac{1}{4|E(H)|} \right),
     \]
    then delete $uv$ from $G_{j}^i$ for some $j \in I$.

    We first claim that this deletion yields a $(H,\psi)$-free graph system.
    Suppose that $\mathcal{G}^i$ contains a rainbow copy of $(H, \psi)$ after deleting edges.
    Let $v_1, \ldots, v_t$ be the vertices of $H$ and for each $p\in [t]$, let $u_p$ be the vertex of $\mathcal{G}^i$ that corresponds to $v_p$ in the copy of $(H, \psi)$.
    Then for each $v_pv_q \in E(\mathrm{sim}(H))$, we have
     \[
      \mu(S_I \cap (J_{u_p} \times J_{u_q})) \geq \frac{1}{n_i^2}(1-\frac{1}{4|E(H)|}),
     \]
    where $I = \psi(E_{v_pv_q})$. Thus
    \begin{align*}
     \int_{[0, 1]^t} \prod_{v_pv_q \in E(\mathrm{sim}(H))}\mathbbm{1}_{S_{\psi(E_{v_pv_q})}}(x_p,x_q)\prod_{p\in [t]} dx_p
     & \geq \int_{J_{u_1} \times \cdots \times J_{u_t}} \prod_{v_pv_q \in E(\mathrm{sim}(H))}\mathbbm{1}_{S_{\psi(E_{v_pv_q})}}(x_p,x_q)\prod_{p\in [t]} dx_p\\
     & \geq \frac{1}{n_i^t} - \sum_{v_pv_q \in E(\mathrm{sim}(H))} \frac{1}{n_i^{t-2}} \mu((J_{u_p} \times J_{u_q})\setminus S_{\psi(E_{v_pv_q})}) \\
     & \geq \frac{1}{n_i^t} - |E(H)| \cdot \frac{1}{n_i^t}\cdot \frac{1}{4|E(H)|} = \frac{3}{4n_i^t} >0.
    \end{align*}
    On the other hand, we have $t^{*}_{(H, \psi)}(\mathbf{W})=0$.
    This implies that
     \[
      \prod_{v_pv_q \in E(\mathrm{sim}(H))} W_{\psi(E_{v_pv_q})}(x_p, x_q) = 0,
     \]
    for almost every $(x_1, \ldots, x_t) \in [0, 1]^t$.
    Hence $$\int_{[0, 1]^t} \prod_{v_pv_q \in E(\mathrm{sim}(H))}\mathbbm{1}_{S_{\psi(E_{v_pv_q})}}(x_p,x_q) \prod_{p\in [t]} dx_p=0,$$ which is a contradiction.
    Therefore, $\mathcal{G}^i$ becomes $(H, \psi)$-free after deletion.

    We now show that there are at most $\varepsilon n_i^2$ edges deleted.
    Let $e_I$ be the number of edges $uv \in \bigcap_{j \in I} G_{j}^i$ such that
     \[
      \mu(S_I \cap (J_u \times J_v)) < \frac{1}{n_i^2} \left( 1-\frac{1}{4|E(H)|} \right).
     \]
    Then
    $$
    \frac{\varepsilon}{2^{k+2}|E(H)|} > \int_{[0, 1]^2} (1-\mathbbm{1}_{S_I})G_{I}^i dxdy \geq e_I \frac{1}{n_i^2}\frac{1}{4|E(H)|},
    $$
    which gives $e_I < \varepsilon n_i^2 / 2^{k}$.
    Therefore, summing over all choices of $I\subseteq [k]$, we have deleted at most $\varepsilon n_i^2$ edges.
    This contradicts the choice of $\mathcal{G}^i$, which completes the proof.
\end{proof}


\section{Induced density}\label{sec:induced_density}

In this section, we introduce the induced density of graphon systems and prove the inverse counting lemma (\Cref{thm:inverse_counting}). 
We note that results in this section are not necessary for our main results but they give additional reasoning as to why introducing a set index is a more natural way to define the graphon system. 

\begin{definition}
    For a graph system $\mathcal{G}=(G_1, \ldots, G_k)$ on $V$ and a coloring tuple $(H, \psi)$, an \emph{induced copy} of $(H, \psi)$ in $\mathcal{G}$ is a graph $H'$ on a vertex set $V(H')\subseteq V$ together with a bijection $\phi:V(H) \rightarrow V(H')$ such that $\phi(u)\phi(v) \in E(G_i)$ if and only if $i \in \psi(E_{uv})$.
\end{definition}
Note that the induced copy of $(H, \psi)$ is only defined for coloring tuples, i.e., $\textbf{dom}(\psi)=E(H)$.
We define the induced homomorphism density for a graphon system.
This generalizes the definition in \cite{Lovasz2012} as $\overline{W}_{1}$ and $\overline{W}_{\emptyset}$ correspond to $W$ and $(1-W)$, respectively, when the graphon system is a graphon, i.e. it has the order $1$.
\begin{definition}
    For a coloring tuple $(H, \psi)$ with $|V(H)|=t$ and a graphon system $\mathbf{W}=(W_I)_{I \subseteq [k]}$ of order $k$, the \textit{induced homomorphism density} of $(H, \psi)$ is defined by 
    $$t^{*}_{\mathrm{ind}, (H, \psi)}(\mathbf{W}) = \int_{[0, 1]^t} \prod_{1 \leq i < j \leq t} \overline{W}_{\psi(E_{ij})}(x_i,x_j) \prod_{i\in [t]} dx_i.$$
\end{definition}
Since each $\overline{W}_I$ can be written as a linear combination of $W_J$, the continuity of this function comes from \Cref{thm:counting_lemma_for_decorated}.
\begin{lemma}\label{lem:induced_density}
    For every coloring tuple $(H, \psi)$, we have 
    $$t^{*}_{\mathrm{ind}, (H, \psi)}(\mathbf{W}) = t^{*}_{ (H, \psi)}(\mathbf{W}) - \sum_{(H', \psi')} t^{*}_{\mathrm{ind}, (H', \psi')}(\mathbf{W}),$$
    where the sum is taken over all coloring tuples $(H', \psi')$ satisfying 
    $E(H') \supsetneq E(H), V(H')=V(H)$ and $\psi'|_{E(H)} = \psi$.
\end{lemma}
\begin{proof}
    We have
    \begin{align*}
        t^{*}_{ (H, \psi)}(\mathbf{W}) & = \int_{[0, 1]^{|V(H)|}} \prod_{ uv\in \binom{V(H)}{2}} W_{\psi(E_{uv})}(x_u,x_v) \prod_{v\in V(H)} dx_v \\
        & = \int_{[0, 1]^{|V(H)|}} \prod_{ uv\in \binom{V(H)}{2}}  \sum_{J \supseteq \psi(E_{uv})} \overline{W}_{J}(x_u,x_v) \prod_{v \in V(H)} dx_v \\
        & = \sum_{(H', \psi')} \int_{[0, 1]^{|V(H)|}}\prod_{ uv\in \binom{V(H)}{2}} \overline{W}_{\psi'(E(H')_{uv})}(x_u,x_v) \prod_{v\in V(H)} dx_v \\
        & = \sum_{(H', \psi')}  t^{*}_{\mathrm{ind}, (H', \psi')}(\mathbf{W}). 
    \end{align*} 
    Again, the sum is taken over all coloring tuples $(H', \psi')$ satisfying 
    $E(H') \supsetneq E(H), V(H')=V(H)$ and $\psi'|_{E(H)} = \psi$. This proves the lemma.
\end{proof}

We now prove the inverse counting lemma using the notion of induced density.
\begin{theorem}[Inverse Counting Lemma]\label{thm:inverse_counting}
    For a sufficiently large $n$ and two systems of graphons $\mathbf{W}=(W_I)_{I \subseteq [k]}$ and $\mathbf{U} = (U_I)_{I \subseteq [k]}$, if the inequality $|t^{\ast}_{(F, \psi)}(\mathbf{W})-t^{*}_{(F, \psi)}(\mathbf{U})| < 2^{-n^2k}$ holds for every coloring tuple $(F, \psi)$ with $|V(F)|=n$, then 
    $\delta_{\square}(\mathbf{W}, \mathbf{U}) < \frac{2000 \times 8^k}{\sqrt{\log n}}$.
\end{theorem}
\begin{proof}
    By repeatedly applying Lemma~\ref{lem:induced_density}, we can express $t^{*}_{\mathrm{ind}, (F, \psi)}(\mathbf{W})$ as a sum of at most $(2^k)^{{n \choose 2}}$ terms of $t^{*}_{(F', \psi')}(\mathbf{W})$ for some coloring tuples $(F', \psi')$ extending $(F, \psi)$. 
    So, we have
    $$|t^{*}_{\mathrm{ind}, (F, \psi)}(\mathbf{W})-t^{*}_{\mathrm{ind}, (F, \psi)}(\mathbf{U})| < 2^{-n^2k} \times (2^k)^{{n \choose 2}} = 2^{-k{n+1 \choose 2}},$$
    for every coloring tuple $(F, \psi)$ with $|V(F)|=n$.

    Consider a coloring tuple $(F, \psi)$ as a graph system of order $k$ on $n$ vertices in the natural way.
    From this point of view, define
    $$\mathcal{F}_{\mathbf{W}} = \left\{(F, \psi) \,\left\mid\, |V(F)|=n,\, \delta_{\square}(\mathbf{W}, (F, \psi))< \frac{1000 \times 8^k}{\sqrt{\log k}} \right.\right\},$$
    and similiarly define $\mathcal{F}_{\mathbf{U}}$.
    If $\mathcal{F}_{\mathbf{W}} \cap \mathcal{F}_{\mathbf{U}} \neq \emptyset$, we are done by the triangle inequality.
    Suppose on the contrary that $\mathcal{F}_{\mathbf{W}} \cap \mathcal{F}_{\mathbf{U}} = \emptyset$.
    Then by \Cref{thm:W_random_graph},
    $$\sum_{(F, \psi) \in \mathcal{F}_{\mathbf{W}}} t^{*}_{\mathrm{ind}, (F, \psi)}(\mathbf{W}) = \sum_{(F, \psi) \in \mathcal{F}_{\mathbf{W}}} \mathbb{P}(\mathbb{G}(n, \mathbf{W}) = (F, \psi)) \geq 1-o(1),$$
    and
    $$\sum_{(F, \psi) \in \mathcal{F}_{\mathbf{W}}} t^{*}_{\mathrm{ind}, (F, \psi)}(\mathbf{U}) \leq \sum_{(F, \psi) \notin \mathcal{F}_{\mathbf{U}}} t^{*}_{\mathrm{ind}, (F, \psi)}(\mathbf{U}) \leq o(1).$$
    Thus for sufficiently large $n$,
    $$\sum_{(F, \psi) \in \mathcal{F}_{\mathbf{W}}} \left( t^{*}_{\mathrm{ind}, (F, \psi)}(\mathbf{W}) - t^{*}_{\mathrm{ind}, (F, \psi)}(\mathbf{U}) \right) \geq \frac{1}{2}.$$
    As a consequence, there exists $(F, \psi) \in \mathcal{F}_{\mathbf{W}}$ such that
    $$t^{*}_{\mathrm{ind}, (F, \psi)}(\mathbf{W}) - t^{*}_{\mathrm{ind}, (F, \psi)}(\mathbf{U}) \geq \frac{1}{2|\mathcal{F}_{\mathbf{W}}| }\geq 2^{-k{n \choose 2}-1} >  2^{-k{n+1 \choose 2}},$$
    which is a contradiction.
\end{proof}
Combining \Cref{thm:inverse_counting} with \Cref{thm:counting_lemma}, one can conclude that a sequence of graphon systems $\mathbf{W}^n$ converges to $\mathbf{W}$ in $\delta_{\square}$-norm if and only if $t_{(F, \psi)}(\mathbf{W}^n)$ converges to $t_{(F, \psi)}(\mathbf{W})$ for every pre-coloring tuple $(F, \psi)$.


\end{document}